\documentclass[12pt, twoside]{article}

\usepackage{amsmath,amsthm,amssymb,color,cite}
\usepackage{times}
\usepackage{enumerate}

\newtheorem{theorem}{Theorem}[section]

\newtheorem{lemma}{Lemma}[section]
\newtheorem{remark}{Remark}[section]
\newtheorem{proposition}{Proposition}[section]

\numberwithin{equation}{section}

\theoremstyle{definition}

\numberwithin{equation}{section}

\newcommand{\subjclass}[1]{\bigskip\noindent\emph{2010 Mathematics Subject Classification:}\enspace#1}
\newcommand{\keywords}[1]{\noindent\emph{Keywords:}\enspace#1}

\begin{document}


\baselineskip=17pt


\title{Well-posedness of non-isentropic Euler equations with physical vacuum}

\author{Yongcai Geng\\
\small {Department of Mathematics, Shanghai Institute of Technology,
Shanghai 201418, P. R. China.}\\
\small {yongcaigeng@alumni.sjtu.edu.cn.}\\[5pt]
{Yachun Li}\\
{\small{School of Mathematical Sciences, MOE-LSC, and SHL-MAC,} }\\{ \small{Shanghai Jiao Tong
University, Shanghai 200240, P. R. China.}}\\
{\small {ycli@sjtu.edu.cn}}\\[5pt]
Dehua Wang\\
\small{Department of Mathematics, University of Pittsburgh,
Pittsburgh, PA 15260, U.S.A.}\\
\small {dwang@math.pitt.edu.}\\[5pt]
Runzhang Xu\\
\small{College of Mathematical Sciences, Harbin Engineering University, Harbin 150001, P. R.  China.}\\
\small {xurunzh@hrbeu.edu.cn.}}

\date{}

\maketitle


\begin{abstract}
We consider the local well-posedness of the one-dimensional non-isentropic compressible Euler equations with moving physical vacuum boundary condition. The physical vacuum singularity requires the sound speed to be scaled as the square root of the distance to the vacuum boundary. The main difficulty lies in the fact that the system of hyperbolic conservation laws becomes characteristic and degenerate at the vacuum boundary.
Our proof is based on an approximation of the Euler equations by a degenerate parabolic
regularization obtained from a specific choice of a degenerate artificial viscosity
term. Then we construct the solutions to this degenerate parabolic problem and
 establish the  estimates that are uniform with respect to the artificial viscosity parameter. Solutions to the compressible Euler equations are obtained as  the limit of the vanishing artificial viscosity.
 Different from the isentropic case \cite{Coutand4, Lei}, our momentum equation of conservation laws   has an extra term $p_{S}S_\eta$ that  leads to   some extra  terms in the energy function  and causes more difficulties even for the case of $\gamma=2$.  Moreover, we deal with this free boundary problem starting from the general cases of $2\leq\gamma<3$ and $1<\gamma<2 $ instead of only emphasizing the isentropic case of $\gamma=2$ in \cite{Coutand4, jang1, Lei}.


\end{abstract}

\section{Introduction}
We are concerned with the one-dimensional compressible flow moving inside  a dynamic vacuum boundary governed by the following non-isentropic Euler equations with initial and free boundary conditions:
\begin{equation}
\label{E:1.1}
\begin{cases} 
\rho_t+(\rho u)_\eta=0, & \text{in}\ \ I(t),\\
(\rho u)_t+(\rho u^2+p)_\eta=0, & \text{in}\ \ I(t),\\
S_t+uS_\eta=0, & \text{in}\ \ I(t),\\
(\rho,u,S)|_{t=0}=\big(\rho_0(\eta),u_0(\eta),S_0(\eta)\big), & \text{on}\ \  I(0),\\
\rho=0, & \text{on}\ \   \Gamma(t),\\
\end{cases}
\end{equation}
where $t>0$ is the time variable, $\eta$ is the space variable; the open bounded interval $I(t)\subset \mathbf{R}$ denotes the evolving domain occupied by the gas, and
$I(0)=I=\{\eta\in\mathbf{R}: 0<\eta<1\}$ is the initial spatial domain;
{$\Gamma(t):=\partial I(t)$ denotes the vacuum boundary
that moves with the fluid velocity;} $u$ represents the Eulerian velocity,
and $ \rho, S$ stand for the density and the entropy of the gas, respectively;  $p=p(\rho, S)$ is the pressure satisfying
the equation of state
\begin{equation}
\label{E:1.2}
p=C_\gamma\rho^\gamma e^{S}, \quad 1<\gamma<{3},
\end{equation}
where $C_\gamma$ is the adiabatic constant that is set to 1 in this paper, and $\gamma$ is the adiabatic gas exponent.
The density $\rho$ satisfies the following conditions:
\begin{equation}
\label{E:1.3}
\rho(\eta,t)>0 \ \ \text{in} \ \text{I(t)}\ \ \ \text{and}\ \ \rho(\eta,t)=0 \ \ \text{on} \ \ \Gamma(t).
\end{equation}
The equations  $\eqref{E:1.1}_1-\eqref{E:1.1}_3$ are the conservation laws of mass,  momentum, and energy, respectively;
$\eqref{E:1.1}_4$ is the initial condition  for the density, velocity, and entropy;
and the boundary condition $\eqref{E:1.1}_5$ states
that the density must vanish along the vacuum boundary.


To understand the behavior of a solution near the vacuum $\rho=0$,  similarly to \cite{liu} we take
$$c^2=p_\rho(\rho,S)=\gamma \rho^{\gamma-1}e^S,$$
and rewrite $\eqref{E:1.1}_1-\eqref{E:1.1}_2$ in terms of $u$ and $c$ as
\begin{equation}\begin{cases}\label{E:1.4}\displaystyle
(c^2)_t+(\gamma-1)c^2u_\eta e^{-S}+u(c^2)_\eta=0, \ \ \ \ \ \ \ \,\\[2pt]
\displaystyle
u_t+uu_\eta+\frac{1}{\gamma-1}(c^2)_\eta +\frac{c^2}{\gamma} S_\eta=0. \ \
\end{cases}\end{equation}
The trajectory of the free boundary
$$
\Gamma(\eta,t)=cl\left\{(\eta,t): \rho(\eta,t)> 0\right\}\cap cl\left\{(\eta,t): \rho(\eta,t)= 0\right\}
$$
coincides with the gas particle path
$$
\frac{d\eta(t)}{dt}=u(\eta(t),t).
$$
Thus, $\eqref{E:1.4}_2$ on $\Gamma(\eta,t) $ becomes
$$
\frac{d u}{dt}=-\frac{1}{\gamma-1}(c^2)_\eta .
$$
Generally, the acceleration $du/dt$ of $\Gamma(\eta,t)$ would be finite,    hence  we have
\begin{equation}
\label{E:1.6}
0<\Big|\frac{\partial c^2(\Gamma(\eta,t),t)}{\partial \eta}\Big|<\infty,\ \ \ \ \text{on} \ \ \Gamma(\eta,t),
\end{equation}
which defines a physical vacuum boundary condition (or singularity).  Since $\rho_0>0$ in $I$, this condition implies
that for some positive constant $C$ and $\eta\in I$ near the vacuum boundary $\Gamma=\partial I,$
\begin{equation}
\label{E:1.7}
\rho_0^{\gamma-1}\geq  C\text{dist}(\eta,\Gamma).
\end{equation}
Equivalently, the physical vacuum condition \eqref{E:1.6} implies that for some $\kappa>0$,
\begin{equation}\label{E:1.8}
0<C_\kappa\le\Big|\frac{\partial \rho_0^{\gamma-1}(\eta)}{\partial \eta}\Big|<\infty, \quad \ \ \text{if}\   \text{dist}(\eta,\partial I)\leq \kappa,
\end{equation}
and  
\begin{equation}\label{E:1.9}
\rho_0^{\gamma-1}(\eta)\geq C_\kappa>0,   \quad \ \ \ \text{if}\ \   \text{dist}(\eta,\partial I)\geq \kappa,
\end{equation}
for a constant $C_\kappa$ depending on $\kappa$.

The mathematical analysis of vacuum states   dates back to Lin \cite{Lin} and Liu-Smoller \cite{liu1} for the isentropic gas
dynamics. The definition of physical vacuum was motivated by the study of the Euler equations with damping
in \cite{liu, liu2}, and we refer the reader to  \cite{liu, luo, Lind, HLarma, HaoW, makino} for more discussions.
%
{ At the vacuum boundary, the hyperbolic system of Euler equations becomes degenerate and the characteristic speeds are singular \cite{liu2},
then the classical theory of hyperbolic systems does not apply, thus even the local existence   of smooth solutions with the physical
vacuum boundary is still largely open. }
%
%
When the data are compactly supported, there are three ways to study the problem.
The first approach consists of solving the compressible Euler system in the whole space and it requires that the system
$\eqref{E:1.1}_1$-$\eqref{E:1.1}_2$ holds in the sense of distribution for all $x\in\mathbf{R}^d$ and $t\in[0,T]$, which
is  the strategy   to construct global weak solutions in \cite{chen,Diperna,Lions}.
The second consists of symmetrizing  the system first and then obtaining the local existence by the theory of symmetric
hyperbolic system  
in the whole space (see for instances \cite{Kato, amj, tms1}
for classical systems and \cite{Geng, Phlippe}  for relativistic systems).

This paper is concerned with the the third one which consists of  requiring
the Euler equations to hold on the set $\{(x,t):\rho(x,t)>0\}$ and writing  an equation for the vacuum boundary
$\Gamma$ that is a free boundary. Here the vacuum boundary is  part of the unknowns. In this case,
an appropriate boundary condition at vacuum is necessary.
{
Suppose that the origin is the initial vacuum contact point and the sound speed $c$ behaves like  $c\sim |x|^h$.
If $h\geq 1$,  the initial contact with vacuum is sufficiently smooth, the local solution  to the Euler equations was obtained in \cite{liu3};
if $0 < h < 1$, the initial contact with vacuum is  H\"{o}lder continuous,  the physical vacuum corresponds to $h=\frac{1}{2}$
(c.f. \cite{Coutand1,Coutand4,Coutand5,jang1}), and the boundary behavior seems ill-posed \cite{jang3}
 when $0 < h < \frac{1}{2}$ and $\frac{1}{2}<h<1$; see \cite{jang2,jang3} for more discussions.
The case $h=0$ occurs in the regime of no continuous initial contact with vacuum.
In this case, a Cauchy problem  can be considered, for example,  the Riemann problem was studied
for the genuinely discontinuous initial data in \cite{Bouchut, Greenspan};
and  a free boundary problem can also be studied, such as   \cite{Lin} for  the positive
density at the vacuum boundary.
}
%
In a series of papers \cite{jang1,jang2,jang3,jang4,jang5} Jang-Masmoudi  gave a rigorous and detailed proof to the existence theory with  physical vacuum boundary.
For the one-dimensional isentropic compressible
gas with  physical vacuum boundary condition, to overcome the degeneracy
difficulty  of propagation speed, in \cite{jang1} they proposed  a new formulation such that some energy
estimates can be closed in the appropriate space, and moreover, they proved that the vacuum boundary behavior
is preserved on some short time interval.
They also investigated the multi-dimensional compressible gas flow with vacuum boundary \cite{jang2}. In \cite{jang4},
 results on free boundary problems were reviewed  and some related open problems were proposed.
Meanwhile, they extended their research to the compressible Navier-Stokes-Poisson system of spherically symmetric   isentropic flows,
and established the local-in-time well-posedness of strong solutions to the vacuum free boundary
problem \cite{jang5}.
Countand-Shkoller \cite{Coutand4,Coutand5} also did many works on this free boundary problem.
In \cite{Coutand4}, they adopted a different approach from Jiang-Masmoudi \cite{jang1} viewing the initial density function $\rho_0$
 as a parameter, thus the isentropic compressible
system becomes a momentum conservation equation, then they used the  vanishing viscosity method to
establish the local existence to the isentropic one-dimensional compressible flow with physical
vacuum. Recently, they also established the {\it a priori} estimates for the free boundary problem of the three-dimensional  compressible Euler equations
\cite{Coutand5}. This technique proposed by
Countand-Shkoller in \cite{Coutand4} has been applied to many other
systems of degenerate and characteristic hyperbolic systems of conservation laws. For example,
Luo-Xin-Zeng$ $ \cite{luo2} and Gu-Lei \cite{Lei} recently extended this method  to the
spherically symmetric system and one-dimensional compressible Euler-Poisson equations with moving
 physical vacuum boundary. Besides the
local existence theory of physical vacuum states, there are also some works on the long time behavior.
Based on self-similar behavior, via Darcy's law, Liu conjectured
\cite{liu} that solutions of Euler equations with damping should behave asymptotically like the solutions of the porous media equation. This problem was  studied by Huang, Marcati and Pan \cite{Huang} in the framework of   entropy solutions
by the method of compensated compactness 
in $L^\infty$, 
and later  by Luo and Zeng  \cite{luo1} by tracking the vacuum boundary.

In this paper, we  deal with the non-isentropic compressible Euler
equations with physical vacuum. Compared with the results already obtained in \cite{Coutand4, Lei, jang1},
there are three novel features.  The first new point is that our momentum equation
has an extra term $p_SS_{\eta}$, because the pressure function depends on not only
the density function $\rho({\eta},t)$ but also the entropy function $S({\eta},t)$ that is equal to $S_0({\eta})$, here $S_0({\eta})$ stands for the initial entropy function (in Lagrangian coordinate, the entropy function satisfies $S_t=0$).
To overcome this difficulty, in the energy function we will introduce more terms like $\sum\limits_{j=1}^2\big|\big|\omega^{2+\mu}\partial_t^{4-2j}\partial_x^{j+2}v(\cdot,t)
\big|\big|_0^2,$ $\big|\big|\omega^{1/2+\mu}\partial_tv'(\cdot,t)
\big|\big|_0^2$ and add a new artificial viscous term $\epsilon(\omega^{2+2\mu}v'e^{S_0})'$, with $'=\frac{\partial }{\partial x}$,
compared with the isentropic case. To close the energy estimates, besides the similar estimates to the isentropic case, we also need to deal with these additional
terms and the term $\epsilon(\omega^{2+2\mu}e^{S_0})'v'$ from the artificial viscous term.
In fact, even for the case of $\gamma=2$, the energy function also has
  more terms  than  the isentropic case.

The second new feature is that we will treat  general $\gamma$ ($2\leq \gamma<3$ and $1<\gamma<2$) from the beginning instead of only emphasizing
the isentropic case $\gamma=2$ as in \cite{Coutand4, jang1, Lei}. Thus, we will face  some new mathematical difficulties. For example, when $2\leq\gamma<3$, we will deal with $\Big|\Big|\frac{1}{\omega^{1/2-\mu}}\Big|\Big|_{L^\beta}$ (c.f  \eqref{E:5.25dda}) instead of $\Big|\Big|\frac{1}{\rho_0^{1/2}}\Big|\Big|_{L^{\beta'}}$ (c.f.  (6.33) in \cite{Coutand4} and (6.45) in \cite{Lei}) for the case of $\gamma=2$ in the process of energy estimates, where determining the proper  value  for $\beta$ is   technically  more difficult than that for $\beta'$
due to $-\frac{1}{4}<\mu\leq0$.

 The third new feature comes from the case
$1<\gamma<2$, since the momentum equation
can be written as the following equation with the distance function $\omega=\rho_0^{\gamma-1}$ as a parameter:
$$
\omega^{\frac{1}{\gamma-1}} v_t+\left(\frac{\rho^\gamma_0e^{S_0}}{(\eta')^\gamma}\right)_x=0,
$$
thus  we know that the coefficient of $v_t$ will degenerate fast as $\gamma\rightarrow1$
near the vacuum boundary. To obtain the $H^2$ norm of $v$ (and thus the $C^1 $norm of $v$) for small $\gamma-1$,
 from the Hardy type embedding inequality \eqref{E:3.1a}, the higher energy function $\tilde{E}(t)$
defined in \eqref{E:2.9b} for $1<\gamma<2$ implies that
$$
||v||_2^2\leq ||v||^2_{\frac{l+1}{2}-\left(\frac{1}{2}+\mu\right)}\leq C\sum_{i=0}^{\frac{l+1}{2}}
\big\vert\big\vert\omega^{1/2+\mu}\partial_x ^i v\big\vert\big\vert_0^2\leq C\tilde{E},\quad  l=3+2\lceil
\frac{1}{2}+\frac{1}{\gamma-1}\rceil,
$$
 indicating the order of derivatives $l\rightarrow \infty$ as $\gamma\rightarrow1$.
 For example, different from the case of $2<\gamma<3$, the estimates of higher order spatial derivatives for the case $\gamma=\frac{3}{2}$ are much more difficult,  for which we use \eqref{E:5.22ab} to close the energy estimates and the details can be found in  \eqref{E:7.19}-\eqref{E:7.30ab}. Finally, we also present the proof of uniqueness for the general case $1<\gamma<3$.

The rest of the paper is organized as follows.
In Section 2, we introduce the Lagrangian coordinates to transform
the free boundary problem to a fixed boundary problem,
and we provide some useful inequalities including the Sobolev embedding inequalities  and
state our main result.  In Section 3, we first present a degenerate  parabolic approximation with viscosity $\varepsilon$
to the compressible Euler equations, then use a fixed point theorem to solve this  approximate problem.
In Section 4 and Section 5, we will prove some uniform estimates  independent of $\varepsilon$ for $2\leq \gamma<3$ and $1<\gamma<2$ respectively. Then we take the limit as $\varepsilon\rightarrow 0$ to obtain the solution of the compressible Euler equations and hence establish  the local existence theorem.
In Section 6, we will prove the main result, i.e., Theorem \ref{T:1.1}.
\bigskip

\section{Preliminaries and main result}

In this section, we provide some preliminaries and state the main result.

\subsection{Lagrangian formulation} 
System \eqref{E:1.1}  is in the Eulerian coordinates $(\eta,t)$. We first rewrite it  in the Lagrangian variables $(x,t)$.
Let $\eta(x,t)$ denote the position of the gas particle $x$
 at time $t$,  then
$$
\partial_t\eta=u(\eta(x,t),t)\ \ \text{for}\ t>0\ \text{and}\ \ \eta(x,0)=x.
$$
Set the Lagrangian velocity,    Lagrangian density  and entropy in the  following:
$$
v(x,t)=u(\eta(x,t),t),\quad f(x,t)=\rho(\eta(x,t),t),\quad \tilde{S}(x,t)=S(\eta(x,t),t).
$$
Then
\begin{equation}\label{E:2.3a}\begin{split}
 v_t &=u_\eta \eta_t+u_t=u_\eta u +u_t,\quad v_x=u_\eta \eta_x,\\
f_t &=\rho_\eta \eta_t+\rho_t=\rho_\eta u+\rho_t,\quad f_x=\rho_\eta \eta_x,\\
\tilde{S}_t &=S_\eta \eta_t+S_t=S_\eta u+S_t,\quad \tilde{S}_x=S_\eta \eta_x.\end{split}\end{equation}
Using \eqref{E:2.3a}, the Lagrangian version of equations $\eqref{E:1.1}_1-\eqref{E:1.1}_3$ can be written on the fixed reference domain $I$ as
\begin{equation}\label{E:2.3}\begin{cases}\displaystyle
f_t+fv_x/\eta_x=0,\\[2pt]
\displaystyle
f v_t+(f^\gamma\exp \tilde{S})_x/\eta_x=0,\\[2pt]
\displaystyle
\tilde{S}_t=0.\\[2pt]
\displaystyle
(f,v,\tilde{S},\eta)|_{t=0}=(\rho_0,u_0,S_0,x).
\end{cases}\end{equation}

From  $\eqref{E:2.3}_3$ we  have 
\begin{equation}\label{E:2.3b}
\tilde{S}(x,t)\equiv S_0(x),
\end{equation}
then \eqref{E:2.3} is reduced to
\begin{equation}\label{E:2.4}\begin{cases}\displaystyle
f_t+fv_x/\eta_x=0,\\[2pt]
\displaystyle
f v_t+(f^\gamma e^{S_0})_x/\eta_x=0,\\[2pt]
\displaystyle
(f,v,\eta)|_{t=0}=(\rho_0,u_0,x).
\end{cases}\end{equation}
From $\eqref{E:2.4}_1$, we know that
$$
(f\eta_x)_t=f_t\eta_x+f\eta_{xt}=f_t\eta_x+fv_x=0,
$$
which implies
$$
f=\rho(\eta(x,t),t)=\rho_0/\eta_x.
$$
Hence, the initial density function $\rho_0$ can be viewed as a parameter in the Euler equations. Thus,
problem \eqref{E:2.4} can be rewritten as
\begin{equation} \label{E:2.7}
\begin{cases}\displaystyle
\rho_0v_t+\left(\frac{\rho_0^\gamma e^{S_0}}{\eta^\gamma_x}\right)_x=0,  &\text{in} \ \ I\times(0,T],\\
\displaystyle
(\eta,v)=(x,u_0),&\text{in}\ \ I\times\{t=0\},\\
\displaystyle
\rho_0^{\gamma-1}=0,& \text{on}\ \ \Gamma,
\end{cases}\end{equation}
with $\rho_0^{\gamma-1}\geq C\text{dist}(x,\Gamma)$ for $x$ near $\Gamma.$ In the following,  we adopt
the notation
\begin{equation} \label{E:2.7a}\begin{split}\omega=\rho_0^{\gamma-1}.
\end{split}\end{equation}
It is obvious that  $\omega=\rho_0$ for $\gamma=2.$

As  noted above, the initial domain $I\subset \mathbf{R}$ at   $t=0$ is
\begin{equation} \label{E:2.8}
I=(0,1),
\end{equation}
and the initial boundary points are   $\Gamma=\partial I=\{0,1\}$.

\subsection{Embedding and interpolation inequalities}
For integers $k\geq 0,$  let the Sobolev space $H^k(I)$  be the completion of $C^\infty(I)$
under the norm
$$
||u||_k:=\left(\sum_{a\leq k}\int_{I} |D^a u(x)|^2dx\right)^{\frac{1}{2}}.
$$
For a real number  $s\geq 0,$ the Sobolev spaces $H^s(I)$ and the norms $||\cdot||_s$ are defined by interpolation.
We use $H_0^1(I)$
to denote the subspace of $H^1(I)$ consisting of those functions $u(x)$  vanishing at $x=0$ and $x=1.$\\

We denote by $||\cdot||_0$  the $L^2$ norm.
We first review some useful embedding and interpolation inequalities.
For Sobolev spaces, one has
\begin{equation}\label{E:3.2a}
||u||_{L^r}\leq C||u||_{\frac{1}{2}}, \quad 1< r<+\infty.
\end{equation}
We will also   use  the   interpolation inequality \cite{danchin}:
\begin{equation}\label{E:3.3}
||u||_{s}\leq C||u||^{1-\theta}_{s_0}||u||^{\theta}_{s_1},
\end{equation}
with $0\leq s_0\leq s\leq s_1$ and $s=(1-\theta)s_0+\theta s_1$ $(0\leq \theta \leq 1)$. In particular, some useful inequalities in this paper are
\begin{equation}\label{E:3.3ab}
||u||_{\frac{3}{4}}\leq C||u||^{\frac{1}{2}}_{\frac{1}{2}}||u||^{\frac{1}{2}}_{1},\quad
||u||_{\frac{1}{2}}\leq C||u||^{\frac{1}{2}}_{0}||u||^{\frac{1}{2}}_{1}.
\end{equation}
For simplicity, we denote by $||\cdot||_\infty$ the  $L^\infty$ norm, then
\begin{equation}\label{E:3.3a}
||u||_{\infty}\leq C_p||u||_1.
\end{equation}
Using \eqref{E:3.3ab}, one has
\begin{equation}\label{E:3.3b}
||u||_{\infty}\leq C_p||u||_{\frac{3}{4}}\leq C||u||^{1/2}_{1/2}||u||^{1/2}_1
\leq C||u||^{1/4}_{0}||u||^{3/4}_1.
\end{equation}

Let $\omega$ denote the distance function to the boundary $\Gamma$. For any $a>0$ and nonnegative integer $b$, the weighted Sobolev space $H^{a,b}(I)$ is given by
 $$H^{a,b}(I):=\{\omega^{a/2}F\in L^2(I):\int_{I}\omega^a(x)|D^k F(x)|^2dx<+\infty,\ \ 0\leq k\leq b \}
$$
with the norm
$$
||F||^{a,b}(I):=\sum\limits_{k=0}^b\int_{I} \omega^a(x)|D^k F(x)|^2dx.
$$
Then for any $b>\frac{a}2,$ one has the following embedding \cite{Kufner}:
\begin{equation}\label{E:3.1a}
H^{a,b}(I)\hookrightarrow H^{b-a/2}(I).
\end{equation}

For the estimates on the higher order spatial derivatives of $v$, we introduce the following lemma.
\begin{lemma}[Lemma 1 in \cite{Coutand2}] \label{L:3.2}  
Let $\varepsilon>0$ and $g\in L^\infty(0, T; H^s(I))$
be given, and let $f\in H^1(0,T;H^s(I))$ be such that
$$
f+\frac{\varepsilon}{\gamma}f_t=g\ \ \ \ \text{in}\ (0,T)\times I.
$$
Then,
$$
||f||_{L^\infty(0,T;H^s(I))}\leq C\max\{||f(0)||_s,||g||_{L^\infty(0,T;H^s(I))}\}.
$$
\end{lemma}

\subsection{Higher-order energy functions and main result}\label{S:2.2}
To close the energy estimates and state the main theorem, we define the   energy functions
for the two cases of $2\leq\gamma<3$ and $1<\gamma<2$, respectively.
We consider the following higher-order energy functions:\\
{\bf Case I: } $2\leq \gamma<3.$
\begin{equation} \label{E:2.9a}\begin{split}
\hat{E}(t)=&\big\vert\big|\omega^{1+\mu}\partial_t^4v'(\cdot,t)\big\vert\big|^2_0+\big\vert\big|
\omega^{1+\mu}\partial_t^4v(\cdot,t)\big\vert\big|^2_0\\
&+\sum_{j=1}^2\Big\{\big\vert\big|\omega^{3/2+\mu}\partial_t^{5-2j}
\partial_x^{j+1}v(\cdot,t)\big|\big\vert^2_0
+\sum_{i=1}^j\big|\big\vert\omega^{1/2+\mu}
\partial_t^{5-2j}\partial_x^iv(\cdot,t)\big|\big\vert^2_0\Big\}\\
&+\sum_{j=1}^2\Big\{\big\vert\big|\omega^{2+\mu}\partial_t^{4-2j}
\partial_x^{j+2}v(\cdot,t)\big|\big\vert^2_0
+\sum_{i=-1}^j\big|\big\vert\omega^{1+\mu}
\partial_t^{4-2j}\partial_x^{i+1}v(\cdot,t)\big|\big\vert^2_0\Big\},
\end{split}\end{equation}
where $-\frac{1}{4}<\mu=\frac{2-\gamma}{2(\gamma-1)}\leq 0$.\\

\noindent
{\bf Case II: } $1<\gamma< 2.$
\begin{equation} \label{E:2.9b}\begin{split}
\tilde{{E}}(t)=&\big\vert\big|\omega^{1+\mu}\partial_t^lv'(\cdot,t)\big\vert\big|^2_0
+\big\vert\big|\omega^{1+\mu}\partial_t^lv(\cdot,t)\big\vert\big|^2_0\\
&+\sum_{j=1}^{\frac{l+1}{2}}\Big\{\big\vert\big|\omega^{3/2+\mu}\partial_t^{l+1-2j}
\partial_x^{j+1}v(\cdot,t)\big|\big\vert^2_0
+\sum_{i=1}^j\big|\big\vert\omega^{1/2+\mu}
\partial_t^{l+1-2j}\partial_x^iv(\cdot,t)\big|\big\vert^2_0\Big\}\\
&+\sum_{j=1}^{\frac{l-1}{2}}\Big\{\big\vert\big|\omega^{2+\mu}\partial_t^{l-2j}
\partial_x^{j+2}v(\cdot,t)\big|\big\vert^2_0
+\sum_{i=-1}^j\big|\big\vert\omega^{1+\mu}
\partial_t^{l-2j}\partial_x^{i+1}v(\cdot,t)\big|\big\vert^2_0\Big\},
\end{split}\end{equation}
where $$\mu=\frac{2-\gamma}{2(\gamma-1)}\geq 0,\quad l=3+2\lceil1/2+\mu\rceil,$$
and $\lceil\cdot\rceil$ is the following ceiling function of  $q\geq 0$:
$$
\lceil q\rceil:=\min\{m: m\geq q, m\ \text{is}\ \text{an}\ \text{integer}\}.
$$
\begin{remark}\label{R:1.1a}
From the definitions of $\mu, l$ in the {\bf case II}, we find that {$\gamma\rightarrow 1$  implies
$l\rightarrow \infty$}. Thus, for  $1<\gamma<2$, the higher-order derivatives will be needed,  which will cause  technical difficulties
  compared with the case of  $\,\,2\leq \gamma<3.$
\end{remark}

\begin{remark}\label{R:1.1} Even when  $\gamma=2,$ the higher-order energy function in
{\bf Case I} is different from the isentropic case in \cite{Coutand4, Lei}, and the former contains
the latter. In the non-isentropic case, some additional terms like $\sum\limits_{j=1}^2\big\vert\big|\rho_0^{2}\partial_t^{4-2j}
\partial_x^{j+2}v(\cdot,t)\big|\big\vert^2_0, \big|\big|{\rho_0^{1/2}}\partial_t v'\big|\big|_0^2$ will appear due to the additional term $p_SS_\eta$
 even for $\gamma=2$.
\end{remark}
In fact, when $\gamma=2$, the energy function for the isentropic case in \cite{Coutand4, Lei} is  
\begin{equation} \label{E:2.9}\begin{split}
E_v(t)=&\sum_{s=0}^4\big\vert\big|\partial_t^sv(\cdot,t)\big\vert\big|^2_{{2-s/2}}+\sum_{s=0}^2\big\vert\big|\rho_0
\partial_t^{2s}v(\cdot,t)\big\vert\big|^2_{{3-s}}\\
&+\big\vert\big|{\rho^{3/2}_0}\partial_t\partial_x^3 v(\cdot,t)\big\vert\big|^2_{0}+\big\vert\big|\rho^{3/2}_0\partial^3_t\partial^2_x
 v(\cdot,t)\big\vert\big|^2_{0}\\
&+\big\vert\big|{\rho^{1/2}_0}\partial_t\partial_x^2 v(\cdot,t)\big\vert\big|^2_{0}+\big\vert\big|\rho^{1/2}_0\partial^3_t\partial_x
 v(\cdot,t)\big\vert\big|^2_{0}.\\
\end{split}\end{equation}

\noindent
Since $\gamma=2$ implies $\mu=0$ and $\omega=\rho_0$,   the energy function in {\bf Case I} is
\begin{equation} \label{E:2.9ab}\begin{split}
\hat{E}(t)=&\big\vert\big|\rho_0\partial_t^4v'(\cdot,t)\big\vert\big|^2_0+
\big\vert\big|\rho_0\partial_t^4v(\cdot,t)\big\vert\big|^2_0\\
&+\sum_{j=1}^2\Big\{\big\vert\big|\rho_0^{3/2}\partial_t^{5-2j}
\partial_x^{j+1}v(\cdot,t)\big|\big\vert^2_0
+\sum_{i=1}^j\big|\big\vert\rho_0^{1/2}
\partial_t^{5-2j}\partial_x^iv(\cdot,t)\big|\big\vert^2_0\Big\}\\
&+\sum_{j=1}^2\Big\{\big\vert\big|\rho_0^{2}\partial_t^{4-2j}
\partial_x^{j+2}v(\cdot,t)\big|\big\vert^2_0
+\sum_{i=-1}^j\big|\big\vert\rho_0
\partial_t^{4-2j}\partial_x^{i+1}v(\cdot,t)\big|\big\vert^2_0\Big\}.
\end{split}\end{equation}
From the weighted Sobolev embedding inequality \eqref{E:3.1a}, we know that
$$
\big|\big|\rho_0\partial_t^4 v'\big|\big|_0^2\geq ||\partial_t^4 v||^2_{1-1}=||\partial_t^4 v||^2_0,
$$
$$
\sum_{j=1}^2\big\vert\big|\rho_0^{2}\partial_t^{4-2j}
\partial_x^{j+2}v(\cdot,t)\big|\big\vert^2_0\geq\sum_{j=1}^2\big\vert\big|\partial_t^{4-2j}
v(\cdot,t)\big|\big\vert^2_{j+2-2}=\sum_{j=1}^2\big\vert\big|\partial_t^{4-2j}
v(\cdot,t)\big|\big\vert^2_{j},
$$
and
$$
\sum_{j=1}^2\big\vert\big|\rho_0^{3/2}\partial_t^{5-2j}
\partial_x^{j+1}v(\cdot,t)\big|\big\vert^2_0\geq \sum_{j=1}^2\big\vert\big|\partial_t^{5-2j}
v(\cdot,t)\big|\big\vert^2_{j-1/2}.
$$
Thus, we have
$$
\sum_{s=0}^4\vert|\partial_t^sv(\cdot,t)\vert|^2_{{2-s/2}}\leq \hat{E}(t).
$$
When $j=1, 2$,  the first term on the second line in \eqref{E:2.9ab} deduces to
the second line of \eqref{E:2.9}. When $j=2, i=2,3$, the second term in the second line in
\eqref{E:2.9ab} leads to the third line of \eqref{E:2.9}. Comparing \eqref{E:2.9} with
\eqref{E:2.9ab}, the latter has some   additional terms  $\sum_{j=1}^2\big\vert\big|\rho_0^{2}\partial_t^{4-2j}
\partial_x^{j+2}v(\cdot,t)\big|\big\vert^2_0,$ $\big|\big|\omega^{1/2+\mu}\partial_t v'\big|\big|_0^2$
caused by $p_SS_\eta$ in the non-isentropic case. The detailed analysis can be seen from the higher-order spatial derivative estimates in Subsection 4.2.

For simplicity, we introduce
\begin{equation} \label{E:2.10}
E(t)=\begin{cases}\hat{E}(t),\quad 2\leq \gamma<3,\\
\tilde{E}(t),\quad 1<\gamma<2.
\end{cases}
\end{equation}

We now state our main result as follows.
\begin{theorem}\label{T:1.1}
Given the initial data $(\rho_0,u_0, S_0)$ with  $E(0)<\infty, \rho_0(x)>0$ in $I$, if the physical vacuum
condition \eqref{E:1.7}  holds for $\rho_0,$  and
\begin{equation}
\label{E:1.3a}
\underline{S}\leq S_0'(x)\leq \bar{S},
\end{equation}
in $I$ for some positive constants $\underline{S}$ and $\bar{S}$,
there exists a unique solution to \eqref{E:2.7} (and hence \eqref{E:1.1}) on $[0,T]$
for some $T>0$   sufficiently small,  such that,
\begin{equation}\label{E:2.11}
\sup\limits_{t\in[0,T]}{E}(t)\leq P(E(0)), 
\end{equation}
where $P$ is  some polynomial function of its argument. 
\end{theorem}

We note that we shall use $P(\cdot)$ to denote a generic polynomial function of its argument, and $P$ will change from line to line with no  explicit expressions
necessarily given in the paper.

\bigskip


\section{The degenerate parabolic approximation of the system}


For   convenience, we write
$'=\frac{\partial }{\partial x}$.
Now for $\varepsilon>0,$ we consider the following nonlinear degenerate parabolic approximation of   \eqref{E:2.7}: 
\begin{equation}\label{E:3.10a}\begin{split}
\rho_0 v_t+\left(\frac{\rho^\gamma_0e^{S_0}}{(\eta')^\gamma}\right)'&=\varepsilon(\rho^{\gamma}_0v'e^{S_0})'
\quad \text{in}\ \ I\times[0,T].\\
\end{split}\end{equation}
With $\rho_0^{\gamma-1}=\omega\geq C\text{dist}(x,\Gamma)$ for $x\in I$ near $\partial I,$ $\eqref{E:3.10a}$ with the initial and boundary conditions becomes
\begin{equation}\label{E:3.10}
\begin{cases}\displaystyle
\omega^{1+2\mu} v_t+\left(\frac{\omega^{2+2\mu}e^{S_0}}{(\eta')^\gamma}\right)'=\varepsilon(\omega^{2+2\mu}v'e^{S_0})',
& \text{in}\ \ I\times[0,T],\\ \displaystyle
(v,\eta)=(u_0,x),  & \text{in}\ I\times\{t=0\},\\ \displaystyle
\rho_0=0, & \text{on}\ \ \Gamma,
\end{cases}\end{equation}
where $\mu=\frac{2-\gamma}{2(\gamma-1)}$. Compared with the isentropic case \cite{Coutand4, Lei},
this nonlinear equation has an additional term 
$\varepsilon\omega^{2+2\mu}v'e^{S_0}S_0'$,  which will cause more technical
difficulties in the following proofs and computations.

Given $u_0$ and $\rho_0$ and using the fact that $\eta(x,0)=x, \eta'(x,0)=1$ we can compute the quantity $v_t|_{t=0}$
for the degenerate parabolic $\varepsilon-$problem \eqref{E:3.10} by using $\eqref{E:3.10a}$:
\begin{equation}\label{E:4.1}\begin{split}
v_t|_{t=0}&=\left(-\frac{1}{\rho_0}\left(\left(\frac{\rho_0}{\eta'}\right)^{\gamma}e^{S_0} \right)'+{\varepsilon}\frac{1}{\rho_0}(\rho^{\gamma}_0v'e^{S_0})'
\right)\Bigg\vert_{t=0}\\
&=-\omega e^{S_0} S_0'+\frac{\gamma}{\gamma-1}\omega'(\varepsilon u_0'-1)e^{S_0}+{\varepsilon}\omega(u_0''+u_0'S_0')e^{S_0}.
\end{split}\end{equation}
Similarly, for all $k\in \mathbb{N},$
\begin{equation}\label{E:4.2}
u_k:=\partial_t^k v|_{t=0}=\partial_t^{k-1}\left(-\frac{1}{\rho_0}\left(\left(\frac{\rho_0}{\eta'}\right)^{\gamma}e^{S_0} \right)'+
\frac{\varepsilon}{\rho_0}(\rho^{\gamma}_0v'e^{S_0})'
\right)\Bigg\vert_{t=0}.
\end{equation}
These formulas show that each $\partial_t^k v|_{t=0}$ is a function of spatial derivatives of $u_0$ and $\rho_0$.

As in \cite{Coutand3, Coutand4, Lei}, the linearized problem of \eqref{E:3.10} is
\begin{equation*}\begin{cases}\displaystyle
\omega^{1+2\mu} v_t+\left(\frac{\omega^{2+2\mu}e^{S_0}}{(\bar{\eta}')^\gamma}\right)'=\varepsilon(\omega^{2+2\mu}v'e^{S_0})',
& \text{in}\ \ I\times[0,T],\\ \displaystyle
(v,\eta)=(u_0,x), & \text{in}\ I\times\{t=0\},\\ \displaystyle
\rho_0=0, & \text{on}\ \ \Gamma,
\end{cases}\end{equation*}
where $\bar{\eta}=x+\int_0^t \bar{v}(x,\tau)d\tau$ and $\bar{v}$ is given,
  the existence and the uniqueness of the solution $v^\varepsilon$ can be obtained by the standard arguments
for the above degenerate parabolic problem \eqref{E:3.10} in a time interval $[0,T^\varepsilon]$ with sufficiently smooth initial data,  using the fixed point argument. Henceforth, we assume that
$T^\varepsilon>0$ is sufficiently small such that,  independent of the choice of $v^\varepsilon$,
\begin{equation}\label{E:4.6}
\eta^\varepsilon(x,t)=x+\int_0^tv^\varepsilon(x,s)ds
\end{equation}
is injective for $t\in[0,T^\varepsilon]$, and   $\frac{1}{2}\leq \eta'(x,t)\leq \frac{3}{2}$ for $t\in[0,T^\varepsilon]$ and $x\in\bar{I}.$

{
We will establish the \textit{ a priori} estimates uniform in $\varepsilon$ to show that the time of existence does not depend
on $\varepsilon$, and then take the weak limit as $\varepsilon\rightarrow 0$ of the sequence of solutions
to \eqref{E:3.10} to obtain the existence of  solution to the Euler system \eqref{E:1.1}.
}





\section{Uniform estimates of $v^\varepsilon$  for $2\leq\gamma<3$}

Our objective in this section is to prove the uniform estimates of $v^\varepsilon$ for $2\leq\gamma<3$.
For the sake of notational convenience, we omit the superscript $\varepsilon$.
We first give some analysis
on $\hat{E}(t)$ in \eqref{E:2.9a} in order to establish the desired estimates.
From the weighted Sobolev embedding inequality \eqref{E:3.1a} and $-\frac{1}{4}\leq\mu\leq0$ as $2\leq\gamma<3$, we have,
for the first term in \eqref{E:2.9a}, 
\begin{equation}\label{E:5.1a}\begin{split}
\big|\big|\omega^{1+\mu}\partial_t^{4}
v'\big|\big|^2_0\geq ||\partial_t^{4}
v||^2_{ 1-1- \mu}\geq ||\partial_t^{4}
v||^2_0,
\end{split}\end{equation}
for the first term of the second line in \eqref{E:2.9a}, 
\begin{equation}\label{E:5.1b}\begin{split}
 \sum_{j=1}^2\big|\big\vert\omega^{3/2+\mu}
\partial_t^{5-2j}\partial_x^{j+2}v(\cdot,t)\big|\big\vert^2_0
\geq\big|\big\vert|
\partial_t^{5-2j}v(\cdot,t)\big|\big\vert^2_{j+2-3/2-\mu}\geq\sum_{j=1}^2 ||\partial_t^{5-2j}v||^2_{j-1/2},
\end{split}\end{equation}
and  for the first term of the third line in \eqref{E:2.9a}, 
\begin{equation}\label{E:5.1c}\begin{split}
\sum_{j=1}^2
\big|\big\vert\omega^{2+\mu}
\partial_t^{4-2j}\partial_x^{i+2}v(\cdot,t)\big|\big\vert^2_{0}\geq\sum_{j=1}^2 ||\partial_t^{4-2j}v||^2_{j+2-2-\mu}
\geq\sum_{j=1}^2 ||\partial_t^{4-2j}v||^2_{j}.
\end{split}\end{equation}
Thus, \eqref{E:5.1a}-\eqref{E:5.1c} lead to
\begin{equation}\label{E:5.1d}\begin{split}
\hat{E}(t)\geq \sum_{s=0}^4
||\partial_t^{s}v||^2_{2-s/2}.
\end{split}\end{equation}

\subsection{Some $\varepsilon-$independent energy estimates on the $\partial_t^k-$problem}

\begin{proposition}\label{L:5.1}
For $2\leq\gamma<3$, there exists a constant $\alpha\in(0,1)$, such that  one has the following $\varepsilon-$independent energy estimate for
$\partial_t^5-$problem of $\eqref{E:3.10}_1$:
 \begin{equation}\label{E:5.2a}\begin{split}
\big|\big|\omega^{1/2+\mu}\partial_t^5 v\big|\big|_0^2+&\big|\big|\omega^{1+\mu}\partial_t^4v'\big|\big|_0^2+\big|\big|\omega^{1+\mu}
\partial_t^4v\big|\big|_0^2
+\varepsilon\int_0^t\big|\big|\omega^{1+\mu}\partial^5_t v'\big|\big|_0^2\\
&\leq \hat{E}^{\alpha}\left(\hat{M}_0+CtP\Big(\sup\limits_{[0,t]}\hat{E}\Big)\right),
\end{split}
\end{equation}
{with  $\hat{M}_0=P(\hat{E}(0))$ and $P(\cdot)$ some polynomial function.}
\end{proposition}
\begin{proof}
We take the fifth partial derivative $\partial_t^5$ in \eqref{E:3.10}    and multiply it by $\partial_t^5 v$, after integrating by parts, we have the following identity:
\begin{equation}\label{E:5.2b}\begin{split}
 \frac{1}{2}\frac{d}{dt}\int_{I} \omega^{1+2\mu}|\partial_t^5 v|^2-\int_I\partial_t^5\left(\frac{\omega^{2+2\mu}}{(\eta')^\gamma}\right)
\partial_t^5 v'e^{S_0}+\varepsilon\int_I\omega^{2+2\mu}|\partial^5_tv'|^2e^{S_0}=0,\\
\end{split}\end{equation}
where $\omega=\rho_0^{\gamma-1}, -\frac{1}{4}<\mu=\frac{2-\gamma}{2(\gamma-1)}\leq 0.$
We see that the second term on the { left}-hand side of \eqref{E:5.2b} can be written as
\begin{equation}\label{E:5.3}\begin{split}
-\int_I\partial_t^5\left(\frac{\omega^{2+2\mu}}{(\eta')^\gamma}\right)\partial^5_t v'e^{S_0}=&\gamma\int_I\frac{\omega^{2+2\mu}}{(\eta')^{\gamma+1}}
\partial^4_t v'\partial^5_t v'e^{S_0}\\
&+\sum_{ i=1}^4c_{ i}\int_I\omega^{2+2\mu}\partial_t^{ i}\frac{1}{(\eta')^{\gamma+1}}\partial_t^{4-{ i}}v'\partial^5_t v'e^{S_0}\\
=&\frac{\gamma}{2}\frac{d}{dt}\int_I\frac{\omega^{2+2\mu}}{(\eta')^{\gamma+1}}\big|\partial_t^4v'\big|^2e^{S_0}\\
&+\frac{(\gamma+1)\gamma}{2}\int_I
\frac{\omega^{2+2\mu}}{(\eta')^{\gamma+2}}v'\big|\partial_t^4v'\big|^2e^{S_0}\\
&+\sum_{ i=1}^4 c_{ i}\int_I\partial_t^{ i}\frac{1}{(\eta')^{\gamma+1}}\partial_t^{4-{ i}}v'\omega^{2+2\mu}\partial^5_t v'
e^{S_0}.\\
\end{split}\end{equation}
Hence, substituting \eqref{E:5.3} into \eqref{E:5.2b}, integrating it from $0$ to $t$, we find that
\begin{equation}\label{E:5.5a}\begin{split}
\frac{1}{2}\int_{I} &\omega^{1+2\mu}\big|\partial_t^5 v\big|^2+\frac{\gamma}{2}\int_I\big|\partial_t^4v'\big|^2\frac{\omega^{2+2\mu}}
{(\eta')^{\gamma+1}}e^{S_0}
+\varepsilon\int_0^t\int_I\omega^{2+2\mu}\big|\partial^5_t v'\big|^2e^{S_0}\\
=&\frac{1}{2}\int_I\omega^{1+2\mu}\big|\partial_t^5 v_0\big|^2+\frac{\gamma}{2}\int_I\frac{\omega^{2+2\mu}}{(\eta')^{\gamma+1}}\Big|_{t=0}
|\partial_t^4v_0'|^2e^{S_0} \\
&-\frac{(\gamma+1)\gamma}{2}\int_0^t\int_I\frac{\omega^{2+2\mu}}{(\eta')^{\gamma+2}}v'\big|\partial_t^4v'\big|^2e^{S_0} \\
&-\sum_{{i}=1}^4 c_{ i}\int_0^t\int_I\partial_t^{ i}\frac{1}{(\eta')^{\gamma+1}}\partial_t^{4- i}v'\omega^{2+2\mu}
\partial^5_t v'e^{S_0} \\
=&\sum_{j=1}^4I_j.
\end{split}\end{equation}
It is easy to verify that $I_1, I_2$ can be controlled by $\hat{M}_0.$
Now we estimate $I_3, I_4$ on the right hand-side of {\eqref{E:5.5a}}.
\begin{equation}\label{E:5.6}\begin{split}
I_3&=-\frac{(\gamma+1)\gamma}{2}\int_0^t\int_I\frac{\omega^{2+2\mu}}{(\eta')^{\gamma+2}}v'|\partial_t^4v'|^2e^{S_0}
\leq C\int_0^t||v'||_{\infty}
\big\vert\big\vert\omega^{1+\mu}\partial_t^4 v'\big\vert\big\vert^2_{0}\\
&\leq C\int_0^t ||v||_2\big\vert\big\vert\omega^{1+\mu}\partial_t^4 v'\big\vert\big\vert^2_{0}
\leq Ct P\Big(\sup\limits_{[0,t]} \hat{E}\Big),\\
\end{split}\end{equation}
where we have used \eqref{E:2.11}, \eqref{E:5.1d} and $1/2\leq\eta'\leq 3/2$ .
Then,  using integration by part in time, we have
\begin{equation}\label{E:5.9}\begin{split}
I_{4}=&-\sum_{i=1}^4 c_i\int_0^t\int_I\partial_t^i\frac{1}{(\eta')^{\gamma+1}}\partial_t^{4-i}v'
\omega^{2+2\mu}\partial^5_t v'
e^{S_0}\\
=&\underbrace{ \int_0^t\int_I\left(\sum_{i=1}^4c_i\partial_t^i\frac{1}{(\eta')^{\gamma+1}}\partial_t^{4-i}v'\right)_t
\omega^{2+2\mu}
\partial^4_t v'e^{S_0}}_{J}\\
&+\int_I\sum_{i=1}^4c_i\partial_t^i\frac{1}{(\eta')^{\gamma+1}}\partial_t^{4-{i}}v'\omega^{2+2\mu}
\partial^4_t v'\Big\vert_0^te^{S_0}.\\
\end{split}\end{equation}
where
$$J:=\int_0^t \sum_{i=1}^{10}J_idt,\quad \text{and}\ J_i=\int_I R(\eta')j_i\omega^{2+2\mu}\partial_t^4 v' dx,$$
the terms $j_i, i=1,2\cdots ,10 $ are the functions of $v', \partial_t^k v', k=1,2,3,4$,  as in the following:
\begin{equation*}\begin{split}
j_1&=\partial_t^4 v'v', \quad j_2=\partial_t^3v' (v')^2,\quad j_3=\partial_t^3 v'\partial_t v',
\quad j_4=\partial_t^2 v' \partial_t v' v',\quad j_5=(\partial_t^2 v')^2,\\
j_6&=\partial^2_t v'(v')^3,\quad j_7=(\partial_t v')^3,\quad j_8=(\partial_t v')^2(v')^2,
\quad j_9=\partial_t v' (v')^4,\quad j_{10}=(v')^6,\\
\end{split}\end{equation*}
where $R(\eta')$ denotes some power functions of $\eta'.$
We first note that
\begin{equation}\label{E:5.10}\begin{split}
|J_1|\leq C ||v'||_{\infty}\big|\big|R(\eta')\big|\big|_{\infty}
\big|\big|\omega^{1+\mu}\partial_t^4 v'\big|\big|_{0}^2
\leq C P\Big(\sup\limits_{[0,t]}\hat{E}\Big),
\end{split}\end{equation}
where we have used \eqref{E:3.3b} and \eqref{E:5.1d}, which means that
\begin{equation*}\label{E:5.10aa}
\begin{split}
||v'||_\infty\leq C ||v'||^{1/4}_{0}||v'||^{3/4}_1
\leq C\hat{E}.
\end{split}\end{equation*}
For $J_2$, we have
\begin{equation*}\label{E:5.11a}\begin{split}
|J_2|\leq  ||R(\eta')||_\infty\big\vert\big\vert\omega^{1/2+\mu}\partial_t^3 v'\big\vert\big\vert_0||
\omega||^{1/2}_\infty||v'||_\infty^2\big\vert\big\vert
\omega^{1+\mu}\partial_t^4 v'\big\vert\big\vert_0
\leq C P\Big(\sup\limits_{\in[0,t]}\hat{E}\Big).
\end{split}\end{equation*}
Similarly, we have
\begin{equation*}\label{E:5.11}\begin{split}
|J_3|&\leq  ||R(\eta')||_\infty\big\vert\big\vert\omega^{1+\mu}\partial_t^3 v'\big\vert\big\vert_{L^4}||\partial_t v'||_{L^4}\big\vert\big\vert
\omega^{1+\mu}\partial_t^4 v'\big\vert\big\vert_0\\
&\leq C
\big|\big|\omega^{1+\mu}\partial_t^3 v'\big|\big|_{1/2}\big|\big|\partial_t v'\big|\big|_{\frac{1}{2}}
\big\vert\big\vert
\omega^{1+\mu}\partial_t^4 v'\big\vert\big\vert^2_0
\leq C P\Big(\sup\limits_{[0,t]}\hat{E}\Big),
\end{split}\end{equation*}
where we have used the physical vacuum condition \eqref{E:1.8} and
$$
\big\vert\big\vert
\omega^{1/2}\partial_x(\omega^{1+\mu}\partial_t^3 v')\big\vert\big\vert^2_0\leq \big\vert\big\vert
\omega^{3/2+\mu}\partial_t^3 v''\big\vert\big\vert^2_0+\big\vert\big\vert
\omega^{1/2+\mu}\omega'\partial_t^3 v'\big\vert\big\vert^2_0\leq \hat{E}.
$$
Using \eqref{E:3.1a}, we have
\begin{equation}\label{E:5.11bc}\begin{split}
\big|\big|\omega^{1+\mu}\partial_t^3 v'\big|\big|_{1/2}=\big|\big|\omega^{1+\mu}\partial_t^3 v'\big|\big|_{1-\frac{1}{2}}\leq \big\vert\big\vert
\omega^{1/2}\partial_x(\omega^{1+\mu}\partial_t^3 v')\big\vert\big\vert_0\leq \hat{E},
\end{split}\end{equation}
and similarly,
\begin{equation*}\label{E:5.11a2}\begin{split}
|J_4|\leq C\big|\big|R(\eta')\big|\big|_{\infty}\big\vert\big\vert
\omega^{1+\mu}\partial_t^2 v'\big\vert\big\vert_{L^4}||\partial_t v'||_{L^4}||v'||_\infty \big\vert\big\vert
\omega^{1+\mu}\partial_t^4 v'\big\vert\big\vert_0
\leq\hat{M}_0+ Ct P\Big(\sup\limits_{[0,t]}\hat{E}\Big),
\end{split}\end{equation*}
where we have used the fact
$$
\big|\big|\omega^{1+\mu}\partial_t^2 v'\big\vert\big\vert_{1/2}=\Big|\Big|\omega^{1+\mu}u_2'+\int_0^t
\omega^{1+\mu}\partial_t^3 v'\Big|\Big|_{1/2}.
$$
Similarly, $J_5$ and $J_6$ can be estimated as
\begin{equation*}\label{E:5.11b}\begin{split}
|J_5|\leq C\big|\big|R(\eta')\big|\big|_{\infty}\big\vert\big\vert
\omega^{1+\mu}\partial_t^2 v'\big\vert\big\vert^2_{L^4}\big\vert\big\vert
\omega^{1+\mu}\partial_t^4 v'\big\vert\big\vert_0
\leq \hat{M}_0+Ct P\Big(\sup\limits_{[0,t]}\hat{E}\Big),
\end{split}\end{equation*}
\begin{equation*}\label{E:5.11ba}\begin{split}
|J_6|\leq C\big|\big|R(\eta')\big|\big|_{\infty}\big\vert\big\vert
\omega^{1+\mu}\partial^2_t v'\big\vert\big\vert_{0}||v'||^3_\infty\big\vert\big\vert
\omega^{1+\mu}\partial_t^4 v'\big\vert\big\vert_0
\leq Ct P\Big(\sup\limits_{[0,t]}\hat{E}\Big).
\end{split}\end{equation*}
For ${J_7-J_{10}}$, using \eqref{E:3.2a} and \eqref{E:5.1d},  we have
\begin{equation*}\label{E:5.11bab}\begin{split}
|J_7|&\leq C\big|\big|R(\eta')\big|\big|_{\infty}\big\vert\big\vert
\partial_t v'\big\vert\big\vert^3_{L^6}||
\omega||^{1+\mu}_\infty\big\vert\big\vert
\omega^{1+\mu}\partial_t^4 v'\big\vert\big\vert_0\\
&\leq C \big\vert\big\vert
\partial_t v'\big\vert\big\vert^3_{1/2}\big\vert\big\vert
\omega^{1+\mu}\partial_t^4 v'\big\vert\big\vert_0
\leq C P\Big(\sup\limits_{[0,t]}\hat{E}\Big),
\end{split}\end{equation*}
and
\begin{equation*}\label{E:5.11bac}\begin{split}
|J_8|\leq& C\big|\big|R(\eta')\big|\big|_{\infty}\big\vert\big\vert
\partial_t v'\big\vert\big\vert^2_{L^4}||v'||^2_\infty\vert\vert
\omega\vert\vert^{1+\mu}_\infty\big\vert\big\vert
\omega^{1+\mu}\partial_t^4 v'\big\vert\big\vert_0
\leq C P\Big(\sup\limits_{[0,t]}\hat{E}\Big),\\
|J_9|\leq& C\big|\big|R(\eta')\big|\big|_{\infty}||\omega||^{1/2}_\infty||v'||_\infty^4\big\vert\big\vert
\omega^{1/2+\mu}\partial_t v'\big\vert\big\vert_{0}\vert\vert
\omega\vert\vert^{1+\mu}_\infty\big\vert\big\vert
\omega^{1+\mu}\partial_t^4 v'\big\vert\big\vert_0
\leq C P\Big(\sup\limits_{[0,t]}\hat{E}\Big),\\
\end{split}\end{equation*}
and
\begin{equation*}\label{E:5.11bae}\begin{split}
|J_{10}|&\leq C\big|\big|R(\eta')\big|\big|_{\infty}\vert\vert
 v'\vert\vert^6_{L^{12}}\vert\vert
\omega\vert\vert^{1+\mu}_\infty\big\vert\big\vert
\omega^{1+\mu}\partial_t^4 v'\big\vert\big\vert_0
\leq C \big\vert\big\vert
\omega^{1+\mu}\partial_t^4 v'\big\vert\big\vert_0\vert\vert
 v'\vert\vert^6_{1/2}\leq C P\Big(\sup\limits_{[0,t]}\hat{E}\Big).\\
\end{split}\end{equation*}
Next, we treat the second term on the right-hand side of  \eqref{E:5.9}, beginning with the case
of $i=1.$ We see that for $\delta>0$,
\begin{equation*}
\begin{split}
\int_I  \partial_t\frac{1}{\eta'^{(\gamma+1)}}\partial_t^3 v'\omega^{2+2\mu}\partial_t^4 v'\big\vert_0^t
=&\int_I \Big(R(\eta')\partial_t^3 v' v'\omega^{2+2\mu}\partial_t^4 v'\Big)(t)\\
&-\int_I \Big(R(\eta')\partial_t^3 v' v'\omega^{2+2\mu}\partial_t^4 v'\Big)(0),
\end{split}\end{equation*}
and
 \begin{equation}\label{E:5.12a}\begin{split}\int_I \Big(R(\eta')\partial_t^3 v' v'\omega^{2+2\mu}\partial_t^4 v'\Big)(t)\leq& C||v'||_{\infty}\big|\big|\omega^{1+\mu}\partial_t^3 v'\big|\big|_{0}
\big|\big|\omega^{1+\mu}\partial_t^4 v'\big|\big|_0\\
\leq & C\hat{E}(t)^{3/4}\left(M_0+CtP\Big(\sup\limits_{[0,t]}\hat{E}\right).\end{split}\end{equation}
The case when $i=2, 3, 4$ can be estimated in the same fashion.
Using \eqref{E:1.3a} again, there exists a constant $\alpha$, such that all the estimates \eqref{E:5.5a}-\eqref{E:5.12a} together  yield
\begin{equation}\label{E:5.14c}\begin{split}
\big|\big|\omega^{1/2+\mu}\partial_t^5 v\big|\big|^2+&\big|\big|\omega^{1+\mu}\partial_t^4v'\big|\big|^2
+\varepsilon\int_0^t\big|\big|\omega^{1+\mu}\partial^5_t v'\big|\big|^2\\
\leq& \hat{E}^\alpha\left(\hat{M}_0+CtP\Big(\sup\limits_{[0,t]}\hat{E}\Big)\right),\quad
0<\alpha<1.
\end{split}\end{equation}
Since $\partial_t^l v=u_l+\int_0^t \partial_t^{l+1} v,$   we have
$$
\big|\big|\omega^{1+\mu}\partial^l_t v\big|\big|_0^2\leq \hat{M}_0+Ct||\omega||_\infty\big|\big|\omega^{1/2+\mu}\partial_t^{l+1} v\big|\big|_0^2.
$$
Letting $l=4$ and using  first term in the energy estimate \eqref{E:5.14c}, we have
\begin{equation}\label{E:5.36}
\big|\big|\omega^{1+\mu}\partial_t^{4} v\big|\big|_0^2\leq \hat{E}^\alpha\left(\hat{M}_0+CtP\Big(\sup\limits_{[0,t]}\hat{E}\Big)\right).
\end{equation}
Thus, the proof of Proposition \ref{L:5.1} is complete.
\end{proof}

Using the same argument as proving Proposition \ref{L:5.1}, we can consider the $\varepsilon-$independent energy estimates for the $\partial_t^3-$problem and $\partial_t-$problem of $\eqref{E:3.10}$ and  obtain the following  estimates:
\begin{proposition}\label{L:5.3} 
For $2\leq\gamma<3$, there exists a constant $\alpha\in(0,1)$,  such that one has the  following energy estimates uniform in $\varepsilon$:

 \begin{equation*}\label{E:5.2ab}
 \begin{split}
\big|\big|\omega^{1/2+\mu}\partial_t^3 v\big|\big|_0^2+&\big|\big|\omega^{1+\mu}\partial_t^2v'\big|\big|_0^2+\big|\big|\omega^{1+\mu}\partial_t^2v\big|
\big|_0^2
+\varepsilon\int_0^t\big|\big|\omega^{1+\mu}\partial^3_t v'\big|\big|_0^2\leq \hat{E}^\alpha\left(\hat{M}_0+CtP\Big(\sup\limits_{[0,t]}\hat{E}\Big)\right)
\end{split}\end{equation*}
and 
 \begin{equation*}\label{E:5.2ac}\begin{split}
\big|\big|\omega^{1/2+\mu}\partial_t v\big|\big|_0^2+&\big|\big|\omega^{1+\mu}v'\big|\big|_0^2+\big|\big|\omega^{1+\mu}v\big|\big|_0^2
+\varepsilon\int_0^t\big|\big|\omega^{1+\mu}\partial_t v'\big|\big|_0^2
\leq \hat{E}^\alpha\left(\hat{M}_0+CtP\Big(\sup\limits_{[0,t]}\hat{E}\Big)\right).\\
\end{split}\end{equation*}
\end{proposition}

\subsection{Estimates of higher-order spatial derivatives}

Having obtained the uniform energy estimates in Propositions \ref{L:5.1}-\ref{L:5.3}, we can begin our estimates of higher-order spatial derivatives.
 We consider $\partial^{k+1}_t-$problem of $\eqref{E:3.10}$:
\begin{equation}\label{E:5.15a}\begin{split}
\big(\omega^{2+2\mu}\partial^k_t v'e^{S_0}\big)'&-\frac{\varepsilon}{\gamma}\partial_t\big(\omega^{2+2\mu}\partial_t^k v'
e^{S_0}\big)'=g,\\
\end{split}\end{equation}
where
\begin{equation*}\begin{split}
g=&-\frac{1}{\gamma}\omega^{1+2\mu}\partial_t^{k+2}v+\sum_{i=1}^k\left(c_{i}\omega^{2+2\mu}\partial_t^{{i}}
\frac{1}{(\eta')^{\gamma+1}}\partial^{k-{i}}_tv'e^{S_0}\right)'\\
&+\left[\left(1-\frac{1}{(\eta')^{\gamma+1}}\right)\big(\omega^{2+2\mu}
\partial_t^{k}v'e^{S_0}\big)'\right]
-(\gamma+1)\frac{\omega^{2+2\mu}\partial_t^k v'e^{S_0}\eta''}{(\eta')^{\gamma+2}}.\\
\end{split}\end{equation*}
Applying Lemma \ref{L:3.2} directly, we have
$$
\big|\big|\omega^{2+2\mu}\partial_t^k v'\big|\big|_0\leq C\max\{||g||_0,||f(0)||_0\}.
$$
However, this estimate is not good enough to obtain the corresponding estimates
with weights of $\omega^{1/2+\mu},
\omega^{3/2+\mu}, \omega^{1+\mu}, \omega^{2+\mu}$ in the energy function $\hat{E}(t)$. To obtain the desired estimates, we shall reduce  \eqref{E:5.15a}
to some new equations of the form in Lemma \ref{L:3.2} by multiplying \eqref{E:5.15a} by some suitable multipliers.

\begin{proposition}\label{L:5.4} For $2\leq\gamma<3,$ there exists a constant $\alpha\in(0,1)$, such that  one has the
following estimates:
\begin{equation}
\label{E:5.32a}\begin{split}
\sup\limits_{[0,t]}\Big(\big|\big|\omega^{1/2+\mu}\partial_t^3 v'\big|\big|_0^2+\big|\big|\omega^{3/2+\mu}\partial_t^3 v''\big|\big|_0^2\Big)
\leq&
C\big(\hat{E}^{\alpha}+1\big)\left(\hat{M_0}+ CtP\Big(\sup\limits_{[0,t]}\hat{E}\Big)\right)+CtP\Big(\sup\limits_{[0,t]}\hat{E}\Big),
\end{split}\end{equation}
 \begin{equation}
\label{E:5.32abcd}\begin{split}
\sup\limits_{[0,t]}\left(\big|\big|\omega^{1/2+\mu}\partial_t v''\big|\big|_0^2+\big|\big|\omega^{3/2+\mu}
\partial_t v'''\big|\big|_0^2\right)\leq&
C\big(\hat{E}^{\alpha}+1\big)\left(\hat{M_0}+ CtP\Big(\sup\limits_{[0,t]}\hat{E}\Big)\right)+CtP\Big(\sup\limits_{[0,t]}\hat{E}\Big).
\end{split}\end{equation}
\end{proposition}
\begin{proof}
First, choosing $k=3$ in \eqref{E:5.15a}, and multiplying both sides of \eqref{E:5.15a} by
$\omega^{-(1/2+\mu)}$, we have
\begin{equation}\label{E:5.15}\begin{split}
\omega^{-(1/2+\mu)}\big(\omega^{2+2\mu}\partial^3_t v'e^{S_0}\big)'&-
\frac{\varepsilon}{\gamma}\partial_t\Big[\omega^{-(1/2+\mu)}\big(\omega^{2+2\mu}\partial_t^3 v'
e^{S_0}\big)'\Big]\\
=&-\frac{1}{\gamma}\omega^{1/2+\mu}\partial_t^{5}v-(\gamma+1)\frac{\omega^{3/2+\mu}\partial_t^3 v'e^{S_0}\eta''}{(\eta')^{\gamma+2}}\\
&+\sum_{i=1}^3c_i\omega^{-(1/2+\mu)}\left(\omega^{2+2\mu}\partial_t^i
\frac{1}{(\eta')^{\gamma+1}}\partial^{3-i}_tv'e^{S_0}\right)'\\
&+\omega^{-(1/2+\mu)}\left[\left(1-\frac{1}{(\eta')^{\gamma+1}}\right)\big(\omega^{2+2\mu}
\partial_t^{3}v'e^{S_0}\big)'\right].\\
\end{split}\end{equation}
Using Lemma \ref{L:3.2} and fundamental theorem of calculus for the
terms on the right-hand side of \eqref{E:5.15}, we obtain that for any $t\in[0,T^\varepsilon]$,
\begin{equation}\label{E:5.19}\begin{split}
\sup\limits_{[0,t]}\Big|\Big|\omega^{-(1/2+\mu)}\big(\omega^{2+2\mu}\partial_t^3v'e^{S_0}\big)'\Big|\Big|_{0}
\leq&  C\sup\limits_{[0,t]}\big|\big|\omega^{1/2+\mu}\partial^5_t v\big|\big|_{0}+C\sup\limits_{[0,t]}\Bigg|\Bigg|\frac{\omega^{3/2+\mu}\partial_t^3v'e^{S_0}\eta''}{\eta'^{(\gamma+2)}}\Bigg|\Bigg|_{0}\\
&+C\sup\limits_{[0,t]}\sum_{i=1}^3\Bigg|\Bigg|\omega^{-(1/2+\mu)}\left(\omega^{2+2\mu}\partial_t^i
\frac{1}{(\eta')^{\gamma+1}}\partial^{3-i}_tv'e^{S_0}\right)'
\Bigg|\Bigg|_{0}\\
&+C\sup\limits_{[0,t]}\Bigg|\Bigg|\omega^{-(1/2+\mu)}
\left[\left(1-\frac{1}{(\eta')^{\gamma+1}}\right)\big(\omega^{2+2\mu}
\partial_t^{3}v'e^{S_0}\big)'\right]\Bigg|\Bigg|_{0}\\
=:&\sum\limits_{j=5}^8 I_j.\\
\end{split}\end{equation}
We  estimate each term on the right-hand side of \eqref{E:5.19}. Noting that
$
-\frac{1}{4}<\mu\leq0,   \ \text{as}\ 2\leq\gamma<3,
$
and  using the first term of the estimate \eqref{E:5.14c}, one has,  for each $t\in[0,T_\varepsilon]$,
\begin{equation}\begin{split}\label{E:5.14a}
I_5=\sup\limits_{[0,t]}\big|\big|\omega^{1/2+\mu}\partial^5_t v\big|\big|_{0}
\leq C\hat{E}^{\alpha/2}\left(\hat{M}_0+CtP\Big(\sup\limits_{[0,t]}\hat{E}\Big)\right).
\end{split}\end{equation}
The second term $I_6$ on the right-hand side of \eqref{E:5.19} can be estimated as
\begin{equation}\label{E:5.29b}\begin{split}
I_{6}=&\Bigg|\Bigg|\frac{\omega^{3/2+\mu}\partial_t^3 v'\eta''e^{S_0}}{(\eta')^{\gamma+2}}\Bigg|\Bigg|_{0}
\leq C \big|\big|\omega^{3/2+\mu}\partial_t^3 v'\big|\big|_\infty\Big|\Big|\int_0^tv''\Big|\Big|_0\\
\leq& Ct\left(\big|\big|\omega^{1/2+\mu}\omega'\partial_t^3 v'\big|\big|_0+\big|\big|\omega^{3/2+\mu}\partial_t^3 v''\big|\big|_0\right)||v''||_0
\leq C tP\Big(\sup\limits_{[0,t]}\hat{E}\Big).
\end{split}\end{equation}
For the third term, we have,
\begin{equation}\label{E:5.22a}\begin{split}
I_7=&\sum_{i=1}^3\Bigg|\Bigg|\omega^{-(1/2+\mu)}\left(\omega^{2+2\mu}\partial_t^i
\frac{1}{(\eta')^{\gamma+1}}\partial^{3-i}_tv'e^{S_0}\right)'
\Bigg|\Bigg|_{0}\\
\leq& C\big|\big|\omega^{3/2+\mu}
{(v')^3}v''
\big|\big|_{0}
+C\big|\big|\omega^{3/2+\mu}
\big(|{\partial_tv''(v')^2}|+|\partial_t v' v''v'|\big)
\big|\big|_{0}\\
&+C\big|\big|\omega^{3/2+\mu}
\big(|\partial^2_t v''v'|+|\partial_t^2 v'v''|\big)
\big|\big|_{0}
+C\big|\big|\omega^{3/2+\mu}
\partial_t v'\partial_tv''
\big|\big|_{0}\\
&+C\big|\big|\big(\omega^{1/2+\mu}\omega'+\omega^{3/2}+\omega^{3/2+\mu}\eta''\big)|(v')^4|
\big|\big|_{0}\\
&+C\big|\big|\big(\omega^{1/2+\mu}\omega'+\omega^{3/2+\mu}+\omega^{3/2+\mu}\eta''\big)
{|\partial_tv'(v')^2|}
\big|\big|_{0}\\
&+C\big|\big|\big(\omega^{1/2+\mu}\omega'+\omega^{3/2+\mu}+\omega^{3/2+\mu}\eta''\big)
{|\partial^2_t v'v'|}
\big|\big|_{0}\\
&+C\big|\big|\big(\omega^{1/2+\mu}\omega'+\omega^{3/2+\mu}+\omega^{3/2+\mu}\eta''\big)
|(\partial_t v')^2|
\big|\big|_{0}\\
=:&\sum_{i=1}^8 I_{7i},\\
\end{split}\end{equation}
where we have used the fact $1/2\leq\eta'\leq 3/2$  and \eqref{E:2.11}.
To obtain the desired estimates, we shall use the following
form of estimates:
\begin{equation}\label{E:5.23abc}\begin{split}
\sum_{j=1}^2\Big\{\big|\big|\omega^{3/2+\mu}\partial_t^{4-2j}\partial_x^{j+1}v\big|\big|_0^2
&+\sum_{i=1}^j\big|\big|\omega^{1/2+\mu}\partial_t^{4-2j}\partial_x^i v\big|\big|_0^2
+\sum_{i=-1}^j\big|\big|\omega^{1+\mu}\partial_t^{3-2j}\partial_x^{i+1} v\big|\big|_0^2\Big\}\\
\leq & C\sum_{j=1}^2\Big\{\Big|\Big|\omega^{3/2+\mu}u^{(j+1)}_{4-2j}+\int_0^t\omega^{3/2+\mu}\partial_t^{5-2j}
\partial_x^{j+1}v\Big|\Big|_0^2\\
&+\sum_{i=1}^j\Big|\Big|\omega^{1/2+\mu} u^{(i)}_{4-2j}+\int_0^t\omega^{1/2+\mu}\partial_t^{5-2j}\partial_x^i v\Big|\Big|_0^2\\
&+\sum_{i=-1}^j\Big|\Big|\omega^{1+\mu} u^{(i+1)}_{3-2j}+\int_0^t\omega^{1/2+\mu}
\partial_t^{4-2j}\partial_x^{i+1}v\Big|\Big|_0^2\Big\}\\
\leq & \hat{M_0}+CtP\Big(\sup\limits_{[0,t]}\hat{E}\Big).\\
\end{split}\end{equation}

For $I_{71}$, using \eqref{E:3.3b}, we have
\begin{equation}\label{E:5.23a}\begin{split}
I_{71}=&\big|\big|\omega^{3/2+\mu}
{(v')^3v''}\big|\big|_{0}\leq\big|\big|\omega^{3/2+\mu}
{v''}\big|\big|_{\infty}||v'||^3_{L^6}\\
\leq& C\big|\big|\omega^{3/2+\mu}v''
\big|\big|_{1}||v'||^3_{1/2}
\leq \hat{M_0}+C tP\Big(\sup\limits_{[0,t]}\hat{E}\Big),
\end{split}\end{equation}
where we have used the physical vacuum condition \eqref{E:1.8} and \eqref{E:5.23abc}, as well as the following two estimates:
\begin{equation*}\label{E:5.23b}\begin{split}
\big|\big|\omega^{3/2+\mu}
{v''}\big|\big|_{\infty}\leq&\big|\big|\omega^{3/2+\mu}
{v''}\big|\big|_{1}\leq C\big|\big|\omega^{1/2+\mu}\omega'
{v''}+\omega^{3/2+\mu}
{v'''}\big|\big|_{0}\leq \hat{M_0}+C tP\Big(\sup\limits_{[0,t]}\hat{E}\Big)\\
\end{split}\end{equation*}
and
\begin{equation}\label{E:5.23c}\begin{split}
||v'||_{1/2}\leq&\Big|\Big|u_0'+\int_0^t\partial_t v'
\Big|\Big|_{1/2}
\leq C\hat{M_0}+C tP\Big(\sup\limits_{[0,t]}\hat{E}\Big).\\
\end{split}\end{equation}
For $I_{72}$, using \eqref{E:5.23c}, we have
\begin{equation*}\label{E:5.24a}\begin{split}
I_{72}=&\big|\big|\omega^{3/2+\mu}
\big({\partial_tv''(v')^2}+\partial_t v' v''v'\big)
\big|\big|_{0}\\
\leq& C\big|\big|
\omega^{\frac{3}{2}+\mu}\partial_t v''\big|\big|_{L^4}||v'||^2_{L^8}
+C\big|\big|
\omega^{\frac{3}{2}+\mu} v''\big|\big|_{\infty}\big|\big|\partial_tv'\big|\big|_{0}||v'||_\infty\\
\leq& C\big|\big|\omega^{\frac{3}{2}+\mu}\partial_t v''\big|\big|^{1/2}_{0}\big|\big|\omega^{\frac{3}{2}+\mu}\partial_t v''\big|\big|^{1/2}_{1}||v'||^2_{1/2}+C\big|\big|
\omega^{\frac{3}{2}+\mu} v''\big|\big|_{1}\big|\big|\partial_tv'\big|\big|_{0}||v'||^{1/2}_{1/2}||v'||^{1/2}_1\\
\leq& \big(1+\hat{E}^{1/4}\big)\left(\hat{M_0}+ tP\Big(\sup\limits_{[0,t]}\hat{E}\Big)\right),\\
\end{split}\end{equation*}
where we used \eqref{E:5.23abc} to derive the following two estimates
$$
\big|\big|\omega^{\frac{3}{2}+\mu}\partial_tv''\big|\big|^{1/2}_0\leq C
||\omega||_\infty^{\frac{1}{4}}\big|\big|\omega^{1+\mu}
\partial_t v''\big|\big|_0^{1/2}\leq \hat{M_0}+ tP\Big(\sup\limits_{[0,t]}\hat{E}\Big),
$$
\begin{equation}\begin{split}\label{E:5.24abcd}
||\partial_t v'||_{0}\leq C\Big|\Big|u_1'+\int_0^t
\partial^2_t v'\Big|\Big|_{0}\leq\hat{M_0}+C tP\Big(\sup\limits_{[0,t]}\hat{E}\Big),
\end{split}\end{equation}
and used the physical vacuum condition \eqref{E:1.8} to get
$$
\big|\big|
\omega^{3/2+\mu} \partial_tv''\big|\big|_{1}^{1/2}=\big|\big|
\omega^{1/2+\mu}\omega' \partial_tv''+
\omega^{3/2+\mu}\partial_t v'''\big|\big|_{0}^{1/2}
\leq  C\hat{E}^{1/4}.
$$
Using \eqref{E:5.23abc} to estimate $\big|\big|\omega^{3/2+\mu}
{\partial^2_t v''}
\big|\big|_{0}, \big|\big|\omega^{1/2+\mu}\partial_t^2 v'\big|\big|_0$
and $\big|\big|\omega^{1/2+\mu}
{ v''}\big|\big|_{0}$, one has, 
\begin{equation*}\label{E:5.25a}\begin{split}
I_{73}=&\big|\big|\omega^{3/2+\mu}
\big(\partial^2_t v''v'+\partial_t^2 v' v''\big)
\big|\big|_{0}\\
\leq& C\big|\big|\omega^{3/2+\mu}
{\partial^2_t v''}
\big|\big|_{0}||v'||_\infty+C\big|\big|\omega^{1/2+\mu}\partial_t^2 v'\big|\big|_0\big|\big|\omega v''\big|\big|_\infty\\
\leq &\big(\hat{E}^{3/8}+1\big)\left(\hat{M_0}+C tP\Big(\sup\limits_{[0,t]}\hat{E}\Big)\right),\\
\end{split}\end{equation*}
where we have used the fact that
\begin{equation*}\label{E:5.25aaa}\begin{split}
\big|\big|\omega v''\big|\big|_\infty\leq& C||\omega||^{-\mu/4}\big|\big|\omega^{1+\mu} v''\big|\big|^{1/4}_0
\big|\big| v''+\omega v'''\big|\big|^{3/4}_0\\
\leq & C\big|\big|\omega^{1/2+\mu} v''\big|\big|^{1/4}_0
\Big(||v''||_0^{3/4}+||\omega||^{-3\mu/4}\big|\big|\omega^{1+\mu}  v'''\big|\big|_0^{3/4}\big)\\
\leq & \hat{E}^{3/8}\left(\hat{M_0}+CtP\Big(\sup\limits_{[0,t]}\hat{E}\Big)\right).
\end{split}\end{equation*}
For $I_{74}$, using \eqref{E:5.24abcd} and \eqref{E:5.23abc} to estimate $\big|\big|\omega^{1+\mu}\partial_t v''\big|\big|_0$, we have
\begin{equation*}\label{E:5.25db}\begin{split}
I_{74}&=\big|\big|\omega^{3/2+\mu}
\partial_t v'\partial_t v''
\big|\big|_{0}
\leq C\big|\big|
\omega^{3/2+\mu}\partial_tv''
\big|\big|_{\infty}||{\partial_t v'}||_{0}\\
\leq&C\big|\big|
\omega^{3/2+\mu}\partial_tv''
\big|\big|^{1/4}_{0}\big|\big|
\omega^{3/2+\mu}\partial_tv''
\big|\big|^{3/4}_{1}||{\partial_t v'}||_{0}\\
\leq &C||\omega||^{1/8}_\infty\big|\big|\omega^{1+\mu}\partial_t v''\big|\big|^{1/4}_0
\big|\big| \omega^{1/2+\mu}
{\partial_t v''}+\omega^{3/2}\partial_tv'''
\big|\big|^{3/4}_{0}||{\partial_t v'}||_{0}\\
\leq& \hat{E}^{\frac{3}{8}}\left(\hat{M_0}+ CtP\Big(\sup\limits_{[0,t]}\hat{E}\Big)\right).
\end{split}\end{equation*}
For $I_{75},$ using \eqref{E:5.23c} to estimate $||v'||_{L^8}$, we have
\begin{equation*}\label{E:5.25bb}\begin{split}
I_{75}=&\big|\big|\big(\omega^{1/2+\mu}\omega'+\omega^{3/2+\mu}+\omega^{3/2+\mu}\eta''\big)(v')^4
\big|\big|_{0}\\
\leq& C\big|\big|\omega^{1/2+\mu}
\big|\big|_{\infty}||v'||^4_{L^8}+C\Big|\Big|\int_0^t \omega^{1+\mu}
{ v''}
\Big|\Big|_{0}||v'||^4_\infty||\omega||^{3/2}_\infty\\
\leq&\hat{M_0}+ CtP\Big(\sup\limits_{[0,t]}\hat{E}\Big).\\
\end{split}\end{equation*}
Similarly, using \eqref{E:5.23c} and \eqref{E:5.24abcd} again,  one has
\begin{equation*}\label{E:5.25ac}\begin{split}
I_{76}=&\big|\big|\big(\omega^{1/2+\mu}\omega'+\omega^{3/2+\mu}+\omega^{3/2+\mu}\eta''\big)
{\partial_tv'(v')^2}
\big|\big|_{0}\\
\leq& C\big|\big|\omega^{1/2+\mu}
 \partial_tv'\big|\big|_{L^4}||v'||^2_{L^8}+C\Big|\Big|\int_0^t
\omega^{3/2+\mu}{ v''}
\Big|\Big|_{L^4}\big|\big|
 \partial_tv'\big|\big|_{L^4}||v'||^2_\infty\\
\leq& \hat{M_0}+C tP\Big(\sup\limits_{[0,t]}\hat{E}\Big),\\
\end{split}\end{equation*}
where we have used \eqref{E:3.1a},  the method similar  to  \eqref{E:5.11bc} to deal with $\big|\big|\omega^{1/2+\mu}
 \partial_tv'\big|\big|_{1/2}$, and \eqref{E:5.23abc} to estimate $\big|\big|\omega^{1+\mu}\partial_t v'\big|\big|_0$,  as well as
\begin{equation*}\label{E:5.25acd}\begin{split}
\big|\big|\omega^{1/2+\mu}
 \partial_tv'\big|\big|_{L^4}\leq&C \big|\big|\omega^{1/2+\mu}
 \partial_tv'\big|\big|_{1/2}\leq C\big|\big|\omega^{\mu}
 \partial_tv'+\omega^{1+\mu}\partial_tv''\big|\big|_{0}
 \leq \hat{M_0}+C tP\Big(\sup\limits_{[0,t]}\hat{E}\Big)
\end{split}\end{equation*}
and
$$
\Big|\Big|\int_0^t
\omega^{3/2+\mu}{ v''}
\Big|\Big|_{L^4}\leq \Big|\Big|\int_0^t
\omega^{3/2+\mu}{ v''}
\Big|\Big|_{1}
\leq \Big|\Big|\int_0^t
\big(\omega^{1/2+\mu}\omega'{ v''}
+\omega^{3/2+\mu}{ v'''}\big)
\Big|\Big|_{0}\leq CtP\Big(\sup\limits_{[0,t]}\hat{E}\Big).
$$
For $I_{77},$  using \eqref{E:5.23abc} to deal with $\big|\big|\omega^{1/2+\mu}
 \partial^2_tv' \big|\big|_{0}$ and \eqref{E:5.23c}, we have
\begin{equation*}\label{E:5.25bd7}\begin{split}
I_{77}=&\big|\big|\big(\omega^{1/2+\mu}\omega'+\omega^{3/2+\mu}+\omega^{3/2+\mu}\eta''\big)
{\partial^2_t v'v'}
\big|\big|_{0}\\
\leq& C\big|\big|\omega^{1/2+\mu}
 \partial^2_tv'\big|\big|_{0}||v'||_\infty\left(1+\Big|\Big|\int_0^t  v''\Big|\Big|_0\right)\\
\leq& \hat{E}^{1/4}\left(\hat{M_0}+C tP\Big(\sup\limits_{[0,t]}\hat{E}\Big)\right)
+Ct P\Big(\sup\limits_{[0,t]}\hat{E}\Big).\\
\end{split}\end{equation*}
For $I_{78}$, using \eqref{E:5.24abcd}, we have
\begin{equation}\label{E:5.25bd8}\begin{split}
I_{78}=&\big|\big|\big(\omega^{1/2+\mu}\omega'+\omega^{3/2+\mu}+\omega^{3/2+\mu}\eta''\big)
(\partial_t v')^2
\big|\big|_{0}\\
\leq& C\big|\big|\omega^{1/2+\mu}
 \partial_tv'\big|\big|_{\infty}||\partial_tv'||_0+C\big|\big|\omega^{3/2+\mu}
 \partial_tv'\big|\big|_{\infty}||\partial_tv'||_0\\
 &+C\Big|\Big|\int_0^t \omega^{3/2+\mu} v''\Big|\Big|_\infty\big|
 \big|\partial_t v'\big|\big|^2_{L^4}\\
\leq& C\big|\big|\omega^{1/2+\mu}
 \partial_tv'\big|\big|_{3/4}||\partial_tv'||_0+C\big|\big|\omega^{3/2+\mu}
 \partial_t v'\big|\big|_{1}||\partial_tv'||_0\\
 &+C t\big|\big|\omega^{3/2+\mu} v''\big|\big|_1\big|
 \big|\partial_t v'\big|\big|^2_{1/2}\\
 \leq& C\Big|\Big|u_1'+\int_0^t \partial^2_t v'\Big|\Big|_0^{1-\alpha}\Big|\Big|\left(\omega^{1/2}
 \partial_tv'\right)'\Big|\Big|^\alpha_{L^{r}}||\partial_tv'||_0\\
 &+C\left(\big|\big|\omega^{1/2+\mu}\omega'
 \partial_tv'\big|\big|_{0}+\big|\big|\omega^{1/2+\mu}\omega
 \partial_tv''\big|\big|_{0}\right)||\partial_tv'||_0\\
 &+C t \left(\big|\big|\omega^{1/2+\mu}\omega'
v''\big|\big|_{0}+\big|\big|\omega^{3/2+\mu}
 v'''\big|\big|_{0}\right)\big|
 \big|\partial_t v'\big|\big|^2_{1/2}\\
\leq& \Big(\hat{E}^{\frac{\alpha}{2}}+1 \Big)\left(\hat{M_0}+C tP\Big(\sup\limits_{[0,t]}\hat{E}\Big)\right)+C tP\Big(\sup\limits_{[0,t]}\hat{E}\Big),\\
\end{split}\end{equation}
where $\alpha_0<\alpha<1$ with $\frac{3}{4}\leq\alpha_0=\frac{3}{4(1+\mu)}=\frac{3}{2}\big(1-\frac{1}{\gamma}\big)<1$ as $2\leq  \gamma <3$ and $\frac{1}{r}=\frac{3}{2}-\frac{3}{4\alpha}$, and we have used the following estimate:
\begin{equation}\label{E:5.25dda}\begin{split}
\Bigg|\Bigg|\left(\omega^{1/2+\mu}
 \partial_tv'\right)'\Bigg|\Bigg|_{L^{r}}&\leq C\Bigg|\Bigg|\frac{\partial_t v'}{\omega^{1/2+\mu}}\Bigg|\Bigg|_{L^{r}}
+C\big|\big|\omega^{1/2+\mu}
 \partial_tv''\big|\big|_{0}\\
 &\leq C\Big|\Big|\frac{1}{\omega^{1/2-\mu}}\Big|\Big|_{L^{\beta}}||\partial_t v'||_{L^{\beta'}}
+C\big|\big|\omega^{1/2+\mu}
 \partial_tv''\big|\big|_{0}\\
 &\leq C\Big|\Big|\frac{1}{\omega^{1/2-\mu}}\Big|\Big|_{L^{\beta}}||\partial_t v'||_{1/2}
+C||\omega||^{-\mu}_\infty\big|\big|\omega^{1/2+\mu}
 \partial_tv''\big|\big|_{0},\\
\end{split}\end{equation}
where, to ensure $\Big|\Big|\frac{1}{\omega^{1/2-\mu}}\Big|\Big|_{L^{\beta}}$ be meaningful, we need $\beta\big(\frac{1}{2}-\mu\big))<1$, so we choose $\frac{1}{\beta}\in\big(\frac{1}{2}-\mu,\frac{1}{r}\big)$,  and
$
\frac{1}{r}=\frac{1}{\beta}+\frac{1}{\beta'}$, with  $\beta'>1$ and 
$||\partial_t v'||_{L^{\beta'}}\leq||\partial_t v'||_{1/2}$ in \eqref{E:3.2a}.
This analysis is different from the isentropic case of $\gamma=2$.
Substituting \eqref{E:5.23a}-\eqref{E:5.25bd8} into \eqref{E:5.22a}, we have
\begin{equation}\label{E:5.22}\begin{split}
I_7\leq \big(\hat{E}^{\frac{\alpha}{2}}+1\big)\left(\hat{M_0}+ C tP\Big(\sup\limits_{[0,t]}\hat{E}\Big)\right)+ CtP\Big(\sup\limits_{[0,t]}\hat{E}\Big).
\end{split}\end{equation}

Now for the last term on the right-hand side of \eqref{E:5.19}, due to
\begin{equation}\label{E:5.22ab}
1-\frac{1}{\eta'^{(\gamma+1)}}=(\gamma+1)\int_0^t\frac{v'}{\eta'^{(\gamma+2)}},
\end{equation}
it can be estimated as
\begin{equation}\label{E:5.23}\begin{split}
I_8=&\Bigg|\Bigg|\omega^{-(1/2+\mu)}
\left[\left(1-\frac{1}{(\eta')^{\gamma+1}}\right)\big(\omega^{2+2\mu}
\partial_t^{3}v'e^{S_0}\big)'\right]\Bigg|\Bigg|_{0}\\
\leq& C \Bigg|\Bigg|(\gamma+1)\int_0^t\frac{v'}{(\eta')^{\gamma+2}}\Bigg|\Bigg|_{\infty}\Big(\big|\big|
\omega^{1/2+\mu}\partial^3_t v'\big|\big|_{0}||\omega'||_\infty+\Big|\Big|
\omega^{1/2+\mu}\partial^3_t v''\Big|\Big|_{0}||\omega||^{1/2}_\infty\\
&+\Big|\Big|
\omega^{1/2+\mu}\partial^3_t v'\Big|\Big|_{0}||\omega||^{1/2}_\infty\Big)
\leq C tP\Big(\sup\limits_{[0,t]}\hat{E}\Big).\\
\end{split}\end{equation}

%
Substituting \eqref{E:5.14a}, \eqref{E:5.29b}, \eqref{E:5.22} and \eqref{E:5.23} into \eqref{E:5.19} leads to
\begin{equation}\label{E:5.29}\begin{split}
\sup\limits_{t\in[0,t]}&\Big|\Big|\omega^{-(1/2+\mu)}\big(\omega^{2+2\mu}\partial_t^3v'e^{S_0}\big)'\Big|\Big|^2_{0}\\
\leq& C\big(\hat{E}^{\alpha/2}+1\big)^2\left(\hat{M_0}+ CtP\Big(\sup\limits_{[0,t]}\hat{E}\Big)\right)+C tP\Big(\sup\limits_{[0,t]}\hat{E}\Big).
\end{split}\end{equation}
Expanding the left-hand side of \eqref{E:5.29} and using physical vacuum condition \eqref{E:1.8}, we have
\begin{equation*}\label{E:5.31a}\begin{split}
\sup\limits_{t\in[0,t]}\Big|\Big|\omega^{-(1/2+\mu)}\big(\omega^{2+2\mu}\partial_t^3v'e^{S_0}\big)'
\Big|\Big|^2_{0}
\geq&{4(1+\mu)^2}\big|\big|\omega^{1/2+\mu}\omega'\partial_t^3v'e^{S_0}\big|\big|_0^2+\big|\big|\omega^{3/2+\mu}\partial_t^3 v''e^{S_0}\big|\big|_0^2\\
&+\big|\big|\omega^{3/2+\mu}\partial_t^3 v'e^{S_0} S_0'\omega^{-\mu}\big|\big|_0^2\\
&+{4(1+\mu)}\int_I\omega^{2+\mu}\partial_t^3v'\partial_t^3v''\omega'
\exp 2S_0\\
&+4(1+\mu)\int_I\omega^{2+2\mu}|\partial_t^3v'|^2\omega'\exp 2S_0 S_0'\\
&+2\int_I\omega^{3+2\mu}\partial_t^3v'\partial_t^3v''\exp 2S_0\\
\geq& C\big|\big|\omega^{1/2+\mu}\partial_t^3v'\big|\big|_0^2
+C\big|\big|\omega^{3/2+\mu}\partial_t^3v''\big|\big|_0^2\\
&-C(1+||\omega||_\infty)\Big(\big|\big|\omega^{1+\mu}\partial_t^3v'\big|\big|_0^2
+\big|\big|\omega^{1+\mu}\partial_t^3v'\big|\big|_0^2\Big)\\
\geq& C\big|\big|\omega^{1/2+\mu}\partial_t^3v'\big|\big|_0^2
+C\big|\big|\omega^{3/2+\mu}\partial_t^3v''\big|\big|_0^2\\
&-\hat{M}_0-CtP\Big(\sup\limits_{[0,t]}\hat{E}\Big),
\end{split}\end{equation*}
and thus, by \eqref{E:5.29},
\begin{equation}\label{E:5.32}\begin{split}
\sup\limits_{[0,t]}&\Big(\big|\big|\omega^{1/2+\mu}\partial_t^3 v'\big|\big|_0^2+\big|\big|\omega^{3/2+\mu}\partial_t^3 v''\big|\big|_0^2\Big)\\
&\leq
C\big(\hat{E}^{\alpha}+1\big)\left(\hat{M_0}+ CtP\Big(\sup\limits_{[0,t]}\hat{E}\Big)\right)+CtP\Big(\sup\limits_{[0,t]}\hat{E}\Big).
\end{split}\end{equation}

Choosing the multiplier $\omega^{-(1/2+\mu)}$
and letting the first term $\omega^{1+2\mu}\partial_t^k v$ in $g$ be
$\omega^{1+2\mu}\partial_t v'$, using the same method as
that of proving \eqref{E:5.32a}, we can prove \eqref{E:5.32abcd}.
\end{proof}

Similarly,
multiplying both sides of \eqref{E:5.15a} by $\omega^{-\mu}$ and replacing the first term $\omega^{1+2\mu}\partial_t^k v$ in $g$ by
$\omega^{1+2\mu}\partial_t^2 v$, $\omega^{1+2\mu}\partial_t^2 v'$ and $\omega^{1+2\mu}v'$, respectively,  one has
\begin{proposition}\label{L:5.5}
For  $2\leq \gamma<3$, there exists some $\alpha\in (0,1)$, such that the following estimate holds,
\begin{equation*}
\label{E:5.32ab}\begin{split}
&\sup\limits_{[0,t]}\Big(\big|\big|\omega^{1+\mu}\partial_t^2 v'\big|\big|_0^2
+\big|\big|\omega^{2+\mu}\partial_t^2 v''\big|\big|_0^2+\big|\big|\omega^{2+\mu}\partial_t^2 v'''\big|\big|_0^2+\big|\big|\omega^{1+\mu}
\partial_t^2 v''\big|\big|_0^2+\big|\big|\omega^{1+\mu} v''\big|\big|_0^2\\
&\qquad\qquad +\big|\big|\omega^{1+\mu}
v'''\big|\big|_0^2+\big|\big|\omega^{2+\mu} \partial_x^4v\big|\big|_0^2\Big)\\
&\leq
C\big(\hat{E}^{\alpha}+1\big)\left(\hat{M_0}+ tP\Big(\sup\limits_{[0,t]}\hat{E}\Big)\right)+CtP\Big(\sup\limits_{[0,t]}\hat{E}\Big).\\
\end{split}\end{equation*}
\end{proposition}
\bigskip

\section{Uniform estimates of $v^\varepsilon$  for  $1<\gamma<2$}

In this section, we shall establish the uniform estimates for
the case   $1<\gamma<2$.
{
As noted in \cite{luo2},  the constant $\gamma$ affects the   degeneracy rate near the vacuum boundary, since $\rho_0$
looks to be the coefficient   of $\partial_t v$ in $\eqref{E:2.7}_1$ and the physical vacuum condition  indicates
that
$$
\rho_0({\eta})\sim \text{dist}({\eta},\partial I)^{\frac{1}{\gamma-1}},\quad \quad {\eta}\rightarrow \partial I.
$$
 Smaller values of $\gamma$  cause   more degeneracy of  $\eqref{E:2.7}_1$  near the vacuum boundary,
then  higher-order derivatives in the energy function are needed to control the $H^2-$norm   and hence the $C^1-$norm of $v$.
From the embedding inequality \eqref{E:3.1a}, the higher energy function $\tilde{E}(t)$
defined in \eqref{E:2.9b} for $1<\gamma<2$ implies that
\begin{equation*}\label{E:7.2}
||v||_2^2\leq ||v||^2_{\frac{l+1}{2}-\left(\frac{1}{2}+\mu\right)}\leq C\sum_{i=0}^{\frac{l+1}{2}}
\big\vert\big\vert\omega^{1/2+\mu}\partial_x ^i v\big\vert\big\vert_0^2\leq C\tilde{E},\quad  l=3+2\lceil
\frac{1}{2}+\mu\rceil.
\end{equation*}
This suggests  that the higher-order energy function $\tilde{E}$ is appropriate for   the physical
vacuum problem \eqref{E:2.7} when $\gamma\in(1,2)$, and as $\gamma\rightarrow 1$ (which means $\mu\rightarrow\infty$) the estimate of $||v||_2^2$ needs infinite higher-order derivatives.
}

\subsection{Energy estimates}

In order to obtain a series of estimates independent of $\varepsilon,$ we first need some energy estimates as in Propositions
\ref{L:5.1}-\ref{L:5.3}.

\begin{proposition}\label{L:7.1}
For $\gamma\in(1,2)$, we have the following $\varepsilon-$independent energy estimate for $\partial_t^{l+1}-$problem of $\eqref{E:3.10}$,
 \begin{equation*}\label{E:5.2aa}\begin{split}
\big|\big|\omega^{1/2+\mu}\partial_t^{l+1} v\big|\big|^2_0+&\big|\big|\omega^{1+\mu}\partial_t^lv'\big|\big|^2_0+\big|\big|\omega^{1+\mu}
\partial_t^lv\big|\big|_0^2
+\varepsilon\int_0^t\big|\big|\omega^{1+\mu}\partial^{l+1}_t v'\big|\big|_0^2\\
\leq& \tilde{E}^\alpha\left( \tilde{M}_0+CtP\Big(\sup\limits_{[0,t]}{\tilde{E}}\Big)\right),\\
\end{split}\end{equation*}
with  $\tilde{M}_0=P(\tilde{E}(0))$ and $P(\cdot)$ some polynomial function.
\end{proposition}

\begin{proof}
Similar to the derivation of \eqref{E:5.2b},
we first take $(l+1)-$th time derivative of equation $\eqref{E:3.10}_1$, then  multiply it by $\partial_t^{l+1} v$ and integrate
this resulting equation with respect to time and space to get
\begin{equation}\label{E:7.5a}\begin{split}
\frac{1}{2}\frac{d}{dt}\int_I\omega^{1+2\mu} |\partial^{l+1}_t v|^2 -\int_I\partial_t^{l+1}\left(\frac{\omega^{2+2\mu}}{(\eta')^\gamma}\right)\partial_t^{l+1}v'e^{S_0}+\varepsilon \int_I
|\omega^{1+\mu}\partial_t^{l+1}v'|^2e^{S_0}=0.
\end{split}\end{equation}
Using similar argument as in \eqref{E:5.3}  to deal with the second item on the left-hand side of \eqref{E:7.5a}, we get
\begin{equation}\label{E:7.7}\begin{split}
\frac{1}{2}\int_I\omega^{1+2\mu} |\partial^{l+1}_t v|^2 &+\frac{\gamma}{2}\int_I \frac{\omega^{2+2\mu}}{(\eta')^{\gamma+1}}|\partial_t^l v'|^2
e^{S_0}+\varepsilon \int_0^t\int_I|\omega^{1+\mu}\partial_t^{l+1}v'|^2\\
=&\frac{1}{2}\int_I\omega^{1+2\mu} |\partial^{l+1}_t v_0|^2+ \frac{\gamma}{2}\int_I \frac{\omega^{2+2\mu}}{(\eta')^{\gamma+1}}|\partial_t^l v'_0|^2
e^{S_0}\\
 &-\frac{\gamma}{\gamma+1}\int_0^t\int_I\omega^{2+2\mu}\frac{v'}{(\eta')^{\gamma+2}}|\partial_t^l v'|^2e^{S_0}\\
&-\sum_{i=1}^lc_i\int_0^t\int_I\omega^{2+2\mu}\partial_t^i\frac{1}{(\eta')^{\gamma+1}}\partial_t^{l-i} v'\partial_t^{l+1}v'e^{S_0}\\
=:&\sum_{i=1}^4I_i.\\
\end{split}\end{equation}
It is obvious   that
$I_1, I_2$ can be controlled by $\tilde{M}_0.$ Now we estimate $I_3$ and $I_4.$
Similar to \eqref{E:5.6}, we have
\begin{equation}\label{E:7.8}\begin{split}
I_3=&-\frac{\gamma}{\gamma+1}\int_0^t\int_I\omega^{2+2\mu}\frac{v'}{(\eta')^{\gamma+2}}|\partial_t^l v'|^2e^{S_0}\\
\leq& C\int_0^t  ||v'||_\infty\big|\big|\omega^{1+\mu}\partial_t^l v'\big|\big|^2_0
\leq C\int_0^t ||v||_2\big|\big|\omega^{1+\mu}\partial_t^l v'\big|\big|^2_0\\
\leq& Ct P\left(\sup\limits_{[0,t]}\tilde{E}\right).
\end{split}\end{equation}
Using integrating by parts in time, we have
\begin{equation}\label{E:7.9}\begin{split}
I_4=&-\sum_{\alpha=1}^lc_i\int_0^t\int_I\omega^{2+2\mu}\partial_t^i\frac{1}{(\eta')^{\gamma+1}}\partial_t^{l-i} v'\partial_t^{l+1}v'e^{S_0}\\
=&\sum_{i=1}^lc_i\int_0^t\int_I\omega^{2+2\mu}\left(\partial_t^i\frac{1}{(\eta')^{\gamma+1}}\partial_t^{l-i} v'\right)_t\partial_t^{l}v'e^{S_0}\\
&-\sum_{i=1}^lc_i\int_I\omega^{2+2\mu}\partial_t^i\frac{1}{(\eta')^{\gamma+1}}\partial_t^{l-i} v'\partial_t^{l}v'e^{S_0}\Big\vert_0^t\\
=:&I_{41}+I_{42}.\\
\end{split}\end{equation}
After detailed computations, using similar analysis as for \eqref{E:5.10}-\eqref{E:5.11bae} and the definition
$\tilde{E}$ in \eqref{E:2.9b}, $I_{41}$ can be estimated as
\begin{equation}\label{E:7.9a}\begin{split}
I_{41}\leq &C\sum_{i=1}^l\int_0^t\int_I\omega^{1+\mu}\Bigg\vert\left(\partial_t^i\frac{1}{(\eta')^{\gamma+1}}\partial_t^{k-i} v'\right)_t\Bigg\vert\big|\omega^{1+\mu}\partial_t^{l}v'\big|\\
\leq& C\int_0^t\Big(\big|\big|\omega^{2+2\mu}\partial_t^lv' v'\big|\big|_0
+\big|\big|\omega^{2+2\mu}\partial_t^{l-1}v' \partial_tv'\big|\big|_0+\cdots\Big)
\big|\big|\omega^{2+2\mu}\partial_t^lv' \big|\big|_0\\
\leq& \tilde{M}_0+Ct P\left(\sup\limits_{[0,t]}\tilde{E}\right).
\end{split}\end{equation}
To show clearly the idea of proving the estimate \eqref{E:7.9a}, we will take $1<\gamma=\frac{3}{2}<2$ as an example
 in the Subsection 5.2.
Using the similar method to \eqref{E:5.12a}, for $0<\alpha<1$, we have
\begin{equation}\label{E:7.13}\begin{split}
I_{42}
\leq \tilde{E}^{{\alpha}}\left(\tilde{M}_0+C t P\left(\sup\limits_{[0,t]}\tilde{E}\right)\right),
\end{split}\end{equation}
where $\tilde{M}_0=P(\tilde{E}(0))$.
Substituting \eqref{E:7.8}-\eqref{E:7.13} into \eqref{E:7.7}, and using the same derivation as \eqref{E:5.36},
we can prove Proposition \ref{L:7.1}.
%

\end{proof}

\subsection{The case of $\gamma=\frac{3}{2}$}
When $\gamma=\frac{3}{2}$, then $\mu=\frac{1}{2}, l=5.$ The higher-order energy function $\tilde{E}(t)$
in \eqref{E:2.9b} is
\begin{equation} \label{E:2.9bc}\begin{split}
\tilde{{E}}(t)=&\big\vert\big|\omega^{1+\mu}\partial_t^5v'(\cdot,t)\big\vert\big|^2_0+
\big\vert\big|\omega^{1+\mu}\partial_t^5v(\cdot,t)\big\vert\big|^2_0\\
&+\sum_{j=1}^{3}\Big\{\big\vert\big|\omega^{3/2+\mu}\partial_t^{6-2j}
\partial_x^{j+1}v(\cdot,t)\big|\big\vert^2_0
+\sum_{i=1}^j\big|\big\vert\omega^{1/2+\mu}
\partial_t^{6-2j}\partial_x^iv(\cdot,t)\big|\big\vert^2_0\Big\}\\
&+\sum_{j=1}^{2}\Big\{\big\vert\big|\omega^{2+\mu}\partial_t^{5-2j}
\partial_x^{j+2}v(\cdot,t)\big|\big\vert^2_0
+\sum_{i=-1}^j\big|\big\vert\omega^{1+\mu}
\partial_t^{5-2j}\partial_x^{i+1}v(\cdot,t)\big|\big\vert^2_0\Big\}.
\end{split}\end{equation}
We note that \eqref{E:2.9bc} also implies the inequality \eqref{E:5.1d}
which will be usually used in the following elliptic estimates.
Now we prove the  energy estimate:
\begin{proposition}\label{L:7.2}
For $\gamma=3/2,$ there exists a constant $\alpha\in(0,1)$,  such that one has the following $\varepsilon-$independent energy estimate on the
$\partial_t^6-$problem of $\eqref{E:3.10}$,
\begin{equation} \label{E:7.15}  
\begin{split}
\big|\big|\omega^{1/2+\mu}\partial_t^6v\big|\big|_0^2&+\big|\big|\omega^{1+\mu}\partial_t^5 v'\big|\big|_0^2
+\big|\big|\omega^{1+\mu}\partial_t^5 v\big|\big|_0^2
+\varepsilon\int_0^t\big|\big|\omega^{1+\mu}\partial_t^6v'\big|\big|_0^2 \\
&\leq  
\tilde{E}^\alpha\left( \tilde{M}_0+CtP\Big(\sup\limits_{[0,t]}{\tilde{E}}\Big)\right),\\
\end{split}\end{equation}
with  $\tilde{M}_0=P(\tilde{E}(0))$ and $P(\cdot)$ some polynomial function.
\end{proposition}
\begin{proof}
In the above energy estimate for Proposition \ref{L:7.1}, except for
the term $I_{41}$ in \eqref{E:7.9a}, the others are the same as the proof of Proposition \ref{L:7.1}.
Now we focus on the estimates of $I_{41}$. After detailed computations, $I_{41}$ is equal to the sum of the following terms:
$$J:=\int_0^t \sum_{i=1}^{14}J_idt,\quad \text{with}\ J_i=\int_I R(\eta')j_i\omega^{2+2\mu}\partial_t^5 v' dx,$$
where the terms $j_i, i=1,2\cdots ,14 $ are the functions of $v', \partial_t^k v', k=1,2,3,4$ as in the following:
\begin{equation*}\label{E:7.14a}\begin{split}
j_1=&v',\quad j_2=\partial_t^4v'(v')^2,\quad j_3=\partial_t^4v'\partial_t v',\quad j_4=\partial_t^3v'(v')^3,
\quad j_5=\partial_t^3v'\partial_tv'v', \\
j_6=&\partial_t^3v'\partial_t^2v',\quad j_7=\partial_t^2v'(v')^4,\quad j_8=\partial_t^2v'\partial_tv'(v')^2,\quad j_9=\partial_tv'(v')^5,
\quad j_{10}=(\partial_tv')^3v',\\
j_{11}=&(\partial_tv')^2(v')^3,\quad j_{12}=(\partial_tv')^2\partial_t^2v',\quad j_{13}=(v')^7,\quad j_{14}=(\partial_t^2v')^2v'.\\
\end{split}\end{equation*}
For $J_1,$ we have
\begin{equation*}\label{E:7.14b}\begin{split}
|J_1|&\leq C\big|\big|\omega^{1+\mu}\partial_t^5 v'\big|\big|^2_0||v'||_\infty\leq CP\Big(\sup\limits_{[0,t]}\tilde{E}\Big).\\
\end{split}\end{equation*}
For $J_2, J_3, J_4,$ we have
\begin{equation*}\label{E:7.14c}\begin{split}
|J_2|&\leq C\big|\big|\omega^{1+\mu}\partial_t^5 v'\big|\big|_0\big|\big|\omega^{1/2+\mu}\partial_t^4 v'\big|\big|_0||v'||^2_\infty||\omega||^{1/2}_\infty\leq CP\Big(\sup\limits_{[0,t]}\tilde{E}\Big),\\
|J_3|&
\leq C\big|\big|\omega^{1+\mu}\partial_t^5 v'\big|\big|_0\big|\big|\omega^{1+\mu}\partial_t^4v'
\big|\big|_{0}\big|\big|v'\big|\big|^2_{1/2}||v'||_1^{1/2}\leq CP\Big(\sup\limits_{[0,t]}\tilde{E}\Big),\\
|J_4|&\leq C\big|\big|\omega^{1+\mu}\partial_t^5 v'\big|\big|_0\big|\big|
\omega^{1+\mu}\partial_t^3v'\big|\big|_0||v'||^3_{\infty}\leq CP\Big(\sup\limits_{[0,t]}\tilde{E}\Big).\\
\end{split}\end{equation*}
For $J_5,$ we have
\begin{equation*}\label{E:7.14e}\begin{split}
|J_5|&\leq C\big|\big|\omega^{1+\mu}\partial_t^5 v'\big|\big|_0\big|\big|\omega^{1+\mu}\partial_t^3v'
\big|\big|_{L^4}\big|\big|\partial_t v'\big|\big|_{L^4}||v'||_\infty\\
&\leq C\big|\big|\omega^{1+\mu}\partial_t^5 v'\big|\big|_0\big|\big|\omega^{1+\mu}\partial_t^3v'
\big|\big|_{1/2}\big|\big|\partial_t v'\big|\big|_{1/2}||v'||_1\\
&\leq \tilde{M}_0+CtP\Big(\sup\limits_{[0,t]}\tilde{E}\Big),
\end{split}\end{equation*}
where we have used the similar method to \eqref{E:5.11bc} to deal with $||\omega^{1+\mu}\partial_t^3 v'||_{1/2}$.
Similarly, we have
\begin{equation*}\label{E:7.14f}\begin{split}
|J_6|\leq &C \big|\big|\omega^{1+\mu}\partial_t^5v'\big|\big|_0\big|\big|\omega^{1+\mu}\partial_t^3v'\big|\big|_\infty
||\partial_t^2 v'||_0\\
\leq& \big|\big|\omega^{1+\mu}\partial_t^5v'\big|\big|_0\big|\big|\omega^{1+\mu}\partial_t^3v'\big|\big|^{1/2}_{1/2}
\big|\big|\omega^{1+\mu}\partial_t^3v'\big|\big|^{1/2}_{0}||\partial_t^2 v'||_0\\
\leq&\tilde{M}_0+ CtP\Big(\sup\limits_{[0,t]}\tilde{E}\Big),\\
|J_7|\leq& C\big|\big|\omega^{1+\mu}\partial_t^5 v'\big|\big|_0\big|\big|\omega^{1+\mu}\partial_t^2v'
\big|\big|_{0}||v'||^4_{\infty}\leq\tilde{M}_0+  CtP\Big(\sup\limits_{[0,t]}\tilde{E}\Big),\\
|J_8|\leq& C\big|\big|\omega^{1+\mu}\partial_t^5 v'\big|\big|_0\big|\big|\omega^{1+\mu}\partial_t^2v'
\big|\big|_{L^4}\big|\big|\partial_t v'\big|\big|_{L^4}||v'||^2_\infty\\
\leq& C\big|\big|\omega^{1+\mu}\partial_t^5 v'\big|\big|_0\big|\big|\omega^{1/2+\mu}\omega'\partial_t^2v'
+\omega^{3/2+\mu}\partial_t^2v''\big|\big|_{0}\big|\big|\partial_t v'\big|\big|_{1/2}||v'||^2_1\\
\leq&\tilde{M}_0+ CtP\Big(\sup\limits_{[0,t]}\tilde{E}\Big),
\end{split}\end{equation*}
where we have used the fact
$$
\omega^{1+\mu}\partial_t^2 v'=\omega^{1+\mu}u_2'+\int_0^t\omega^{1+\mu}\partial_t^3 v',
$$
and the similar method to \eqref{E:5.11bc} to deal with $||\omega^{1+\mu}\partial_t^2 v'||_{1/2}$.

Using \eqref{E:3.2a}, $J_9, J_{10}, J_{11}, J_{12}, J_{13}$ can be estimated as
\begin{equation*}\label{E:7.14k}\begin{split}
|J_9|&\leq C\big|\big|\omega^{1+\mu}\partial_t^5 v'\big|\big|_0\big|\big|\partial_tv'
\big|\big|_{L^4}||v'||_{L^4}||v'||^4_\infty
\leq \tilde{M}_0+CP\Big(\sup\limits_{[0,t]}\tilde{E}\Big),\\
|J_{10}|&\leq C\big|\big|\omega^{1+\mu}\partial_t^5 v'\big|\big|_0\big|\big|\partial_tv'
\big|\big|^3_{L^6}||v'||^4_\infty
\leq \tilde{M}_0+CP\Big(\sup\limits_{[0,t]}\tilde{E}\Big),\\
|J_{11}|&\leq C\big|\big|\omega^{1+\mu}\partial_t^5 v'\big|\big|_0\big|\big|\partial_tv'
\big|\big|^2_{L^4}||v'||^3_\infty
\leq \tilde{M}_0+CP\Big(\sup\limits_{[0,t]}\tilde{E}\Big),\\
|J_{12}|&\leq C\big|\big|\omega^{1+\mu}\partial_t^5 v'\big|\big|_0\big|\big|\partial_tv'
\big|\big|^2_{L^8}||\omega^{1+\mu}\partial_t^2v'||_{L^4}
\leq \tilde{M}_0+CP\Big(\sup\limits_{[0,t]}\tilde{E}\Big),\\
|J_{13}|&\leq C\big|\big|\omega^{1+\mu}\partial_t^5 v'\big|\big|_0||v'||_0||v'||^6_\infty
\leq CP\Big(\sup\limits_{[0,t]}\tilde{E}\Big).
\end{split}\end{equation*}
For $J_{14},$ we have
\begin{equation*}\label{E:7.14o}\begin{split}
|J_{14}|&\leq C\big|\big|\omega^{1+\mu}\partial_t^5 v'\big|\big|_0\big|\big|\omega^{1+\mu}\partial_t^2 v'\big|\big|_\infty||v'||_\infty||\partial_t^2v'||_0\\
&\leq C\big|\big|\omega^{1+\mu}\partial_t^5 v'\big|\big|_0\big|\big|\omega^{\mu}\omega'\partial_t^2 v'
+\omega^{1+\mu}\partial_t^2 v''\big|\big|_0||v'||_\infty||\partial_t^2v'||_0\\
&\leq CP\Big(\sup\limits_{[0,t]}\tilde{E}\Big).
\end{split}\end{equation*}
Thus, using the same argument as Proposition \ref{L:5.1}, for $0<\alpha<1$, we obtain \eqref{E:7.15}.
\end{proof}
Similarly, we have
\begin{proposition}\label{L:7.3}
For $\gamma=\frac{3}{2}$, there exists a constant $\alpha\in(0,1)$,  such that one has the following $\varepsilon-$independent energy estimates:
\begin{equation}\label{E:7.15a}\begin{split}
\big|\big|\omega^{1/2+\mu}\partial^{4}_t v\big|\big|_0^2 +\big|\big|\omega^{1+\mu}\partial_t^3 v'\big|\big|_0^2
+&\big|\big|\omega^{1+\mu}\partial_t^3 v\big|\big|_0^2
+\varepsilon \int_0^t\big|\big|\omega^{1+\mu}\partial_t^{4}v'\big|\big|_0^2\\
\leq& \tilde{E}^\alpha\left( \tilde{M}_0+C t P\left(\sup\limits_{[0,t]}\tilde{E}\right)\right),
\end{split}\end{equation}
\begin{equation}\label{E:7.15b}\begin{split}
\big|\big|\omega^{1/2+\mu}\partial^{2}_t v\big|\big|_0^2 +\big|\big|\omega^{1+\mu}\partial_t v'\big|\big|_0^2+
&\big|\big|\omega^{1+\mu}\partial_t v\big|\big|_0^2
+\varepsilon \int_0^t\big|\big|\omega^{1+\mu}\partial_t^{2}v'\big|\big|_0^2\\
\leq &\tilde{E}^\alpha\left(\tilde{M}_0+C t P\left(\sup\limits_{[0,t]}\tilde{E}\right)\right).
\end{split}\end{equation}
\end{proposition}

\subsection{Estimates of higher-order spatial derivatives for $\gamma=\frac{3}{2}$}

Based on the energy estimate \eqref{E:7.15}, by elliptic estimates we can derive the estimates of the higher-order spatial derivatives  associated with the weights.

\begin{proposition}\label{L:7.5}
For $\gamma=\frac{3}{2}$, there exists a constant $\alpha\in(0,1)$,
such that one has the following estimates:
\begin{equation}\label{E:7.16}\begin{split}
\sup\limits_{[0,t]}\Big(\big|\big|\omega^{1/2+\mu}\partial_t^4v'\big|\big|_0^2
+\big|\big|\omega^{3/2+\mu}\partial_t^4v''\big|\big|_0^2\Big)
\leq C\big(\tilde{E}^{\alpha}+1\big)\left({\tilde{M}_0}+ CtP\Big(\sup\limits_{[0,t]}{\tilde{E}}\Big)\right),\end{split}\end{equation}
\begin{equation}\label{E:7.16ab}\begin{split}
\sup\limits_{[0,t]}\Big(\big|\big|\omega^{1/2+\mu}\partial_t^2v'\big|\big|_0^2+\big|\big|\omega^{1/2+\mu}\partial_t^2v''\big|\big|_0^2
+&\big|\big|\omega^{1/2+\mu}\partial_t^2v'''\big|\big|_0^2
+\big|\big|\omega^{1/2+\mu}v'\big|\big|_0^2+\big|\big|\omega^{1/2+\mu}v''\big|\big|_0^2\\
+\big|\big|\omega^{1/2+\mu}v'''\big|\big|_0^2+\big|\big|\omega^{3/2+\mu}\partial_x^4v\big|\big|_0^2\Big)\leq& C\big(\tilde{E}^{\alpha}+1\big)\left({\tilde{M}_0}+ CtP\Big(\sup\limits_{[0,t]}{\tilde{E}}\Big)\right).\end{split}\end{equation}
\end{proposition}
\begin{proof}Choosing $k=4$ in \eqref{E:5.15a}, using the first term of \eqref{E:7.15}, multiplying \eqref{E:5.15a} by
$\omega^{-{1/2+\mu}}$,
then, using Lemma \ref{L:3.2} and fundamental theorem of calculus, we obtain that, for any $t\in[0,T^\varepsilon]$,
\begin{equation}\label{E:7.17}\begin{split}
\sup\limits_{[0,t]}\Bigg|\Bigg|\omega^{-(1/2+\mu)}&\big(\omega^{2+2\mu}\partial_t^4v'e^{S_0}\big)'\Bigg|\Bigg|_{0}\leq
C\sup\limits_{[0,t]}\big|\big|\omega^{1/2+\mu}\partial^6_t v\big|\big|_{0}\\
&+C\sup\limits_{[0,t]}\sum_{\alpha=1}^4\Bigg|\Bigg|\omega^{-(1/2+\mu)}\left(\omega^{2+2\mu}\partial_t^\alpha
\frac{1}{(\eta')^{\gamma+1}}\partial^{4-\alpha}_tv'e^{S_0}\right)'
\Bigg|\Bigg|_{0}\\
&+C\sup\limits_{[0,t]}\Bigg|\Bigg|\omega^{-(1/2+\mu)}
\left[\left(1-\frac{1}{(\eta')^{\gamma+1}}\right)\big(\omega^{2+2\mu}
\partial_t^{4}v'e^{S_0}\big)'\right]\Bigg|\Bigg|_{0}\\
&+C\sup\limits_{[0,t]}\Bigg|\Bigg|\frac{\omega^{3/2+\mu}\partial_t^4v'e^S_0\eta''}{\eta'^{(\gamma+2)}}\Bigg|\Bigg|_{0}
=:\sum\limits_{i=1}^4 I_i.\\
\end{split}\end{equation}

Now we estimate each term on the right-hand side of \eqref{E:7.17}. First, we   use
the estimate \eqref{E:7.15} to obtain,  for each $t\in[0,T_k]$,
\begin{equation}\label{E:7.18a}
I_1=\sup\limits_{[0,t]}\big|\big|\omega^{1/2+\mu}\partial^6_t v\big|\big|_{0}\leq \tilde{E}^{{\alpha/2}}\Big({\tilde{M}}_0
+CtP\Big(\sup\limits_{[0,t]}{\tilde{E}}\Big)\Big).
\end{equation}
The remaining terms will be estimated by using the definition of the energy function $\tilde{{E}}.$
For the second term, from \eqref{E:5.22ab}, after detailed computations, we have
\begin{equation}\label{E:7.19}
\begin{split}
I_2=&\sum_{\alpha=1}^4\Bigg|\Bigg|\omega^{-(1/2+\mu)}\left(\omega^{2+2\mu}\partial_t^\alpha
\frac{1}{(\eta')^{\gamma+1}}\partial^{4-\alpha}_tv'e^S_0\right)'
\Bigg|\Bigg|_{0}\\
=&\sum_{\alpha=1}^4\Bigg|\Bigg|-\frac{1}{\gamma+1}\omega^{-(1/2+\mu)}\left(\omega^{2+2\mu}\partial^{4-\alpha}_tv'e^S_0
\int_0^t\partial_t^\alpha\Big(\frac{v'}{(\eta')^{\gamma+2}}\Big)\right)'
\Bigg|\Bigg|_{0} \leq \sum_{i=1}^{74} K_{i},
\end{split}\end{equation}
where we have used the fact $1/2\leq\eta'\leq 3/2$ and $\underline{S}\leq S'_0\leq \overline{S}.$

For $i=1,2,\cdots, 17$, $K_i=C\big|\big|(\omega^{1/2+\mu}+\omega^{3/2+\mu})k_i\big|\big|_0$ with
\begin{equation} \label{kk117}
\begin{split}
k_1=&\partial_t^3v'\int_0^t (v')^2,\  k_2=\partial_t^3v'\int_0^t \partial_t v',\
k_3=\partial_t^2v'\int_0^t(v')^3,\  k_4=\partial_t^2v'\int_0^t v'\partial_t v',\\
k_5=&\partial_t^2v'\int_0^t \partial_t^2 v', \ k_6=\partial_tv'\int_0^t (v')^4, \
k_7=\partial_tv'\int_0^t (v')^2\partial_t v', \ k_8=\partial_tv'\int_0^t (\partial_t v')^2,\\
k_9=&\partial_tv'\int_0^t v'\partial^2_t v', \ k_{10}=\partial_tv'\int_0^t \partial^3_t v',\
k_{11}=v'\int_0^t (v')^5,\ k_{12}=v'\int_0^t (v')^3\partial_t v',\\
k_{13}=&v'\int_0^t (v')^2\partial^2_t v',\ k_{14}=v'\int_0^t v'(\partial_t v')^2,\
k_{15}=v'\int_0^t \partial_tv'\partial^2_tv',\ k_{16}=v'\int_0^t v'\partial^3_t v',\\
k_{17}=&v'\int_0^t \partial_t^4 v'.
\end{split}\end{equation}
For $i=18,19,\cdots, 58$,  $K_i=C\big|\big|\omega^{3/2+\mu}k_i\big|\big|_0$ with 
\begin{equation*}\begin{split}
k_{18}=&\partial_t^3v'\int_0^t  v'v'',\  k_{19}=\partial_t^3v'\int_0^t \partial_t v'',\
k_{20}=\partial_t^2v'\int_0^t  (v')^2v'',\  k_{21}=\partial_t^2v'\int_0^t v''\partial_t v',\\
k_{22}=&\partial_t^2 v'\int_0^t v'\partial_t v'',\ k_{23}=\partial^2_tv'\int_0^t \partial_t^2v'', \
k_{24}=\partial_tv'\int_0^t (v')^3 v'',\ k_{25}=\partial_tv'\int_0^t v'\partial_t v', \\
k_{26}=&\partial_tv'\int_0^t (v')^2\partial_t v'', \ k_{27}=\partial_tv'\int_0^t \partial_tv'\partial_t v'',\
k_{28}=\partial_t v'\int_0^t \partial_t^2v'v'',\ k_{29}=\partial_tv'\int_0^t v'\partial_t^2v'',\\
 k_{30}=&\partial_t v'\int_0^t \partial^3_t v'',\
k_{31}=v'\int_0^t (v')^4v'',\ k_{32}=v'\int_0^t (v')^2v''\partial_t v',\ k_{33}=v'\int_0^t (v')^3\partial_t v'',\\
 k_{34}=&v'\int_0^t v''\partial_t v',\ k_{35}=v'\int_0^t \partial_t v'\partial_t v'',\ k_{36}=v'\int_0^t \partial^2_t v'',
 \ k_{37}=v'\int_0^t v'v''\partial^2_t v',\\
 k_{38}=&v'\int_0^t \partial_t^2v'\partial_t v'',\ k_{39}=v'\int_0^t \partial_t v'\partial^2_t v'',\ k_{40}=v'\int_0^t v'\partial^3_tv'',
k_{41}=v'\int_0^t \partial_t^3v'v'', \\
 k_{42}=&v'\int_0^t \partial_t^4 v'',\ k_{43}=\partial_t^3v''\int_0^t (v')^2,\ k_{44}=\partial^3_t v''\int_0^t \partial_tv',\
 k_{45}=\partial_t^2v''\int_0^t (v')^3, \\
\end{split}\end{equation*}
\begin{equation*}\begin{split}
k_{46}=&\partial^2_tv''\int_0^tv'' \partial_t v',\ k_{47}=\partial_t^2v''\int_0^t \partial^2_t v',\ k_{48}=\partial_tv''\int_0^t (v')^4,
\ k_{49}=\partial_tv''\int_0^t(v')^2 \partial_tv',\\
 k_{50}=&\partial_tv''\int_0^t (\partial_t v')^2,\ k_{51}=\partial_tv''\int_0^t v'\partial^2_t v',\ k_{52}=\partial_tv''\int_0^t \partial^3_tv',
 \ k_{53}=v''\int_0^t (v')^2(\partial_t v')^2,\\
 k_{54}=&v''\int_0^t(v')^2\partial_t^2 v',\ k_{55}=v''\int_0^t \partial_t v'\partial_t^2v',\ k_{56}=v''\int_0^tv'\partial_t^3v',\
k_{57}=v''\int_0^t\partial_t^4v'.\\
\end{split}\end{equation*}
For $i=58,59,\cdots, 74$,  $K_i=C\big|\big|\omega^{3/2+\mu}k_i\big|\big|_0$, where each $k_i$ has the form  similar to
the term  $k_{i-17}$ in \eqref{kk117} but every integrand function has an additional term $\eta''$, for example,
\begin{equation*}\begin{split}
k_{58}=&\partial_t^3v'\int_0^t (v')^2\eta'',\;  k_{59}=\partial_t^3v'\int_0^t \partial_t v''\eta'',\;
\cdots,\;
k_{74}=v'\int_0^t\partial_t^4 v'\eta''.\\
\end{split}\end{equation*}
Before deriving the estimates of  $K_i,i=1,2,\cdots,74$, similarly to \eqref{E:5.23abc}, one has
\begin{equation}\label{E:5.23abcde}\begin{split}
\sum_{j=1}^3\Big\{\big|\big|\omega^{3/2+\mu}\partial_t^{5-2j}\partial_x^{j+1}v\big|\big|_0^2
&+\sum_{i=1}^j\big|\big|\omega^{1/2+\mu}\partial_t^{5-2j}\partial_x^i v\big|\big|_0^2
+\sum_{i=-1}^j\big|\big|\omega^{1+\mu}\partial_t^{4-2j}\partial_x^{i+1} v\big|\big|_0^2\Big\}\\
\leq & \tilde{{M}}_0+CtP\Big(\sup\limits_{[0,t]}\tilde{{E}}\Big).\\
\end{split}\end{equation}
For ${K_{1},K_2, K_4, K_5}$, using the estimates \eqref{E:5.23abcde}, we have
\begin{equation}\label{E:7.20a}\begin{split}
K_{1}=& C\Big|\Big|(\omega^{1/2+\mu}+\omega^{3/2+\mu})\partial_t^3v'\int_0^t (v')^2 \Big|\Big|_0
\leq C\Big|\Big|\omega^{1/2+\mu}\partial_t^3 v'
\Big|\Big|_{0}\int_0^t||v'||^2_{\infty}\\
\leq& {\tilde{M}_0}+ CtP\Big(\sup\limits_{[0,t]}{\tilde{E}}\Big),\\
K_{2}=& C\Big|\Big|(\omega^{1/2+\mu}+\omega^{3/2+\mu})\partial_t^3v'\int_0^t \partial_t v' \Big|\Big|_0
\leq C\Big|\Big|\omega^{1/2+\mu}\partial_t^3v'\Big|\Big|_0\big(||v'||_{\infty}+||u'_0||_{\infty}\big)\\
\leq&C\tilde{E}^{1/2}\big({\tilde{M}_0}+ \tilde{E}^{1/2}\big),\\
K_4=&C\Big|\Big|(\omega^{1/2+\mu}+\omega^{3/2+\mu})\partial_t^2v'\int_0^t v'\partial_tv' \Big|\Big|_0
\leq C\Big|\Big|\omega^{1/2+\mu}\partial_t^2 v'
\Big|\Big|_{L^4}||\partial_t v'||_{L^4}\int_0^t||v'||_\infty\\
\leq& {\tilde{M}_0}+ CtP\Big(\sup\limits_{[0,t]}{\tilde{E}}\Big),\\
K_5=&C\Big|\Big|(\omega^{1/2+\mu}+\omega^{3/2+\mu})\partial_t^2v'\int_0^t \partial^2_tv' \Big|\Big|_0
\leq C\Big|\Big|\omega^{1/2+\mu}\partial_t^2 v'
\Big|\Big|_{\infty}\int_0^t||\partial_t^2v'||_0\\
\leq& CtP\Big(\sup\limits_{[0,t]}{\tilde{E}}\Big),\\
\end{split}\end{equation}
For $K_7, K_8, K_9,$ we have
\begin{equation*}\label{E:7.22}\begin{split}
K_7=&C\Big|\Big|(\omega^{1/2+\mu}+\omega^{3/2+\mu})\partial_tv'\int_0^t (v')^2\partial_tv' \Big|\Big|_0
\leq C\int_0^t\Big|\Big|\omega^{1/2+\mu}\partial_t v'
\Big|\Big|_{L^8}||\partial_tv'||_{L^8}||v'||^2_{L^4}\\
\leq& CtP\Big(\sup\limits_{[0,t]}{\tilde{E}}\Big),\\
K_8=&C\Big|\Big|(\omega^{1/2+\mu}+\omega^{3/2+\mu})\partial_tv'\int_0^t (\partial_tv')^2 \Big|\Big|_0
\leq C\int_0^t\Big|\Big|\omega^{1/2+\mu}\partial_t v'
\Big|\Big|_{L^4}||\partial_tv'||^2_{L^8}\\
\leq& CtP\Big(\sup\limits_{[0,t]}{\tilde{E}}\Big),\\
K_9=&\Big|\Big|(\omega^{1/2+\mu}+\omega^{3/2+\mu})\partial_tv'\int_0^t \partial_t^2 v'v' \Big|\Big|_0
\leq C\int_0^t\big|\big|\omega^{1/2+\mu}\partial_t v'
\big|\big|_{\infty}||\partial^2_tv'||_{0}||v'||_{\infty}\\
\leq& CtP\Big(\sup\limits_{[0,t]}{\tilde{E}}\Big).\\
\end{split}\end{equation*}
{ For $K_{10}$, using the same method as dealing with $I_{78}$ in \eqref{E:5.24abcd} and \eqref{E:5.25dda}, we have
\begin{equation*}\label{E:7.23}\begin{split}
K_{10}=&\Big|\Big|(\omega^{1/2+\mu}+\omega^{3/2+\mu})\partial_tv'\int_0^t \partial_t^3 v' \Big|\Big|_0
\leq C\Big|\Big|\omega
^{1/2+\mu}\partial_t^2 v'
\Big|\Big|_{\infty}||\partial_t v'||_{0}\\
\leq& C\big|\big|\omega
^{1/2+\mu}\partial_t^2 v'\big|\big|_{3/4}||\partial_t v'||_{0}
\leq C\big|\big|\omega
^{1/2+\mu}\partial_t^2 v'\big|\big|^{1-\alpha}_{0}\big|\big|(\omega
^{1/2+\mu}\partial_t^2 v')'\big|\big|^{\alpha}_{L^r}||\partial_t v'||_{0}\\
\leq&  C\big|\big|\omega
^{1/2+\mu}\partial_t^2 v'\big|\big|^{1-\alpha}_{0}\big(\big|\big|\omega
^{\mu-1/2}\partial_t^2 v'\big|\big|^{\alpha}_{L^r}+\big|\big|\omega
^{1/2+\mu}\partial_t^2 v''\big|\big|^{\alpha}_{L^r}\big)||\partial_t v'||_{0}\\
\leq&  C\big|\big|\omega
^{1/2+\mu}\partial_t^2 v'\big|\big|^{1-\alpha}_{0}\big(||\omega^\mu||_\infty\Big|\Big|
\frac{\partial_t^2 v'}{\sqrt{\omega}}\Big|\Big|^{\alpha}_{L^r}+\big|\big|\omega
^{1/2+\mu}\partial_t^2 v''\big|\big|^{\alpha}_{L^r}\big)||\partial_t v'||_{0}\\
\leq&  C\big|\big|\omega
^{1/2+\mu}\partial_t^2 v'\big|\big|^{1-\alpha}_{0}\big(\Big|\Big|
\frac{1}{\sqrt{\omega}}\Big|\Big|^{\alpha}_{L^{\beta}}||\partial_t^2 v'||_0^\alpha+\big|\big|\omega
^{1/2+\mu}\partial_t^2 v''\big|\big|^{\alpha}_{L^{r}}\big)||\partial_t v'||_{0}\\
\leq &C \tilde{E}^{\alpha/2}.\\
\end{split}\end{equation*}
To ensure $\frac{1}{2}\beta<1$, here we choose
$\frac{1}{\beta}\in\big(\frac{1}{2},\frac{1}{r}\big)$, and $\alpha_0<\alpha<1$ with $0\leq \alpha_0=\frac{3}{4(1+\mu)}<\frac{3}{4}$
as $1<\gamma<2$ and $\frac{1}{r}=\frac{3}{2}-\frac{3}{4\alpha}$. }Similarly,
\begin{equation*}\label{E:7.23a}\begin{split}
K_{19}=&\Big|\Big|\omega^{3/2+\mu}\partial^3_tv'\int_0^t \partial_t v'' \Big|\Big|_0
\leq C\big(\Big|\Big|\omega
^{1/2+\mu}v''
\Big|\Big|_{\infty}+\Big|\Big|\omega
^{1/2+\mu}u''_0
\Big|\Big|_{\infty}\big)||\omega^{1/2+\mu}\partial^3_t v'||_{0}\\
\leq& C\tilde{E}^{\alpha/2},\\
K_{52}=&\Big|\Big|\omega^{3/2+\mu}\partial_tv''\int_0^t \partial^3_t v' \Big|\Big|_0
\leq C\Big|\Big|\omega
^{1/2+\mu}\int_0^t\partial^3_tv'
\Big|\Big|_{\infty}||\omega^{1/2+\mu}\partial_t v''||_{0}
\leq C\tilde{E}^{{\alpha/2}}.\\
\end{split}\end{equation*}
For $K_{18}$, we have
\begin{equation*}\label{E:7.24}\begin{split}
K_{18}=\Big|\Big|\omega^{3/2+\mu}\partial_t^3v'\int_0^t  v'v'' \Big|\Big|_0
\leq C\int_0^t\big|\big|\omega
^{3/2+\mu}
{\partial^3_t v''}
\big|\big|_{\infty}|| v''||_{0}||v'||_\infty
\leq {\tilde{M}_0}+C tP\Big(\sup\limits_{[0,t]}{\tilde{E}}\Big).\\
\end{split}\end{equation*}
For $K_{27}$, we have
\begin{equation*}\label{E:7.25}\begin{split}
K_{27}=&\Big|\Big|\omega^{3/2+\mu}\partial_tv'\int_0^t \partial_t v'\partial_tv'' \Big|\Big|_0
\leq C\int_0^t\big|\big|\omega
^{3/2+\mu}
{\partial_t v''}
\big|\big|_{\infty}||\partial_t v'||^2_{L^4}
\leq {\tilde{M}_0}+C tP\Big(\sup\limits_{[0,t]}{\tilde{E}}\Big).
\end{split}\end{equation*}
By detailed analysis, we find that the estimates of $K_i, i=3,6,11,12,13,16,17,31,33,36,40,42,43,\\44,45,48,51,57$ are the same   as $K_1$;
the estimates of $K_i, i=14,15,29,50$
are the same as $K_4$; the estimates of $K_i, i=23,39$
are the same as $K_5$; the estimates of $K_i, i=53,54,55,56$ are the same as
$K_{9}$; the estimates of $K_i, i=20,21,22,24,28,32,34,37,38,41,46,47,49$
are the same as $K_{18}$, thus we omit them.
Due to $\eta''=\int_0^tv''$, similarly to $K_{18}$, we also have
\begin{equation}\label{E:7.30a}\begin{split}
K_{58}=&C\Big|\Big|\omega^{3/2+\mu}\partial_t^3 v'\int_0^t  (v')^2\eta'' \Big|\Big|_0
\leq \int_0^t||v'||^2_\infty\big|\big|\omega^{3/2+\mu}\partial_t^3 v'\big|\big|_\infty
\Big|\Big|\int_0^t v''\Big|\Big|_0\\
\leq&\tilde{M}_0+ C tP\Big(\sup\limits_{[0,t]}{\tilde{E}}\Big),\\
\end{split}\end{equation}
and the estimates of $K_i, i=59,\cdots,74,$ can also be obtained in the same way. Substituting  all the estimates of $K_i$
into \eqref{E:7.19}, one has
\begin{equation}\label{E:7.30ab}\begin{split}
I_2
\leq&\tilde{E}^{{\alpha/2}}\Big(\tilde{M}_0+ C tP\Big(\sup\limits_{[0,t]}{\tilde{E}}\Big)\Big)
+CtP\Big(\sup\limits_{[0,t]}{\tilde{E}}\Big).\\
\end{split}\end{equation}
From \eqref{E:5.22ab},
the third term $I_3$ on the right-hand side of \eqref{E:7.17} can be controlled by
\begin{equation*}\label{E:7.31}\begin{split}
I_3=&\Bigg|\Bigg|\omega
\left[\left(1-\frac{1}{(\eta')^{\gamma+1}}\right)\big(\omega^{2+2\mu}
\partial_t^{4}v'e^S_0\big)'\right]\Bigg|\Bigg|_{0}
\leq CtP\Big(\sup\limits_{[0,t]}{\tilde{E}}\Big).\\
\end{split}\end{equation*}
The last term of \eqref{E:7.17} can be estimated as
\begin{equation}\label{E:7.36}\begin{split}
I_4
\leq C\big|\big|\omega^{2+\mu}\partial_t^4v'\big|\big|_\infty\Big|\Big|\int_0^tv''\Big|\Big|_0
\leq CtP\Big(\sup\limits_{[0,t]}{\tilde{E}}\Big).\\
\end{split}\end{equation}
Substituting \eqref{E:7.18a}, \eqref{E:7.30ab} and \eqref{E:7.36} into \eqref{E:7.17}, we have
\begin{equation*}\label{E:7.37}\begin{split}
\Big|\Big|\omega^{-(1/2+\mu)}&\big(\omega^{2+2\mu}\partial_t^4 v'e^S_0\big)'\Big|\Big|_0^{{2}}
\leq C\big(\tilde{E}^{\alpha}+1\big)\left({\tilde{M}_0}+ tP\Big(\sup\limits_{[0,t]}{\tilde{E}}\Big)\right)
+C tP\Big(\sup\limits_{[0,t]}{\tilde{E}}\Big).\\
\end{split}\end{equation*}
Similarly to   \eqref{E:5.29}-\eqref{E:5.32}, we have
\begin{equation*}\label{E:7.38}\begin{split}
\sup\limits_{[0,t]}\Big(\big|\big|\omega^{1/2+\mu}\partial_t^4 v'\big|\big|^{{2}}_0&+\big|\big|\omega^{3/2+\mu}\partial_t^4 v''\big|\big|^{2}_0\Big)\\
\leq& C\big(1+\tilde{E}^{\alpha}\big)\left({\tilde{M}_0}+ tP\Big(\sup\limits_{[0,t]}{\tilde{E}}\Big)\right)
+CtP\Big(\sup\limits_{[0,t]}{\tilde{E}}\Big).\\
\end{split}\end{equation*}
Choosing the multipliers $\omega^{-(1/2+\mu)}$, and replacing the first term $\omega^{1/2+\mu}\partial_t^{6} v$ on the right hand-side of \eqref{E:7.17}
by $\omega^{1/2+\mu}\partial_t^{4} v,$ $\omega^{1/2+\mu}\partial_t^{4} v',$ $\omega^{1/2+\mu}\partial_t^{2} v,$
$\omega^{1/2+\mu}\partial_t^{2} v',$ $\omega^{1/2+\mu}\partial_t^{2} v'',$ respectively, we can obtain the
estimates \eqref{E:7.16ab}.
\end{proof}

Furthermore,  Choosing the multiplier $\omega^{-\mu}$, and replacing $\omega^{1/2+\mu}\partial_t^{k+2} v$
in \eqref{E:7.17}
by $\omega^{1+\mu}\partial_t^{5} v,$ $\omega^{1+\mu}\partial_t^{5} v',$ $\omega^{1+\mu}\partial_t^{3} v',$
$\omega^{1+\mu}\partial_t^{3} v'',$  respectively, we have the following proposition:
\begin{proposition}\label{L:7.6}
For $\gamma=\frac{3}{2}$, there exists a constant $\alpha\in(0,1)$, such that one has the following estimate:
\begin{equation*}\label{E:7.16aab}\begin{split}
\sup\limits_{[0,t]}\Big(\big|\big|\omega^{1+\mu}\partial_t^3v'\big|\big|_0^2+\big|\big|\omega^{1+\mu}\partial_t^3v''\big|\big|_0^2
+&\big|\big|\omega^{2+\mu}\partial_t^3v'''\big|\big|_0^2+\big|\big|\omega^{1+\mu}\partial_tv''\big|\big|_0^2
+\big|\big|\omega^{1+\mu}\partial_tv'''\big|\big|_0^2\\
+\big|\big|\omega^{2+\mu}\partial_t\partial_x^4v\big|\big|_0^2\Big)
\leq& C\big(\tilde{E}^{\alpha}+1\big)\left({\tilde{M}_0}+ CtP\Big(\sup\limits_{[0,t]}{\tilde{E}}\Big)\right).
\end{split}\end{equation*}
\end{proposition}
\section{The proof of Theorem \ref{T:1.1}}

We are now ready to finish the proof of Theorem \ref{T:1.1} as follows.
From \eqref{E:2.10}, the Propositions \ref{L:5.1}-\ref{L:5.5} for $2\leq \gamma<3$ and Propositions \ref{L:7.2}-\ref{L:7.5} for $\gamma=\frac{3}{2}$,
one has
$$
\sup\limits_{[0,T]}{E}(t)\leq {M}_0+CtP\Big(\sup\limits_{[0,t]}{E}\Big)+C\sup\limits_{[0,t]}{E}^\alpha\left({M}_0
+tP\Big(\sup\limits_{[0,t]}{E}\Big)\right),
$$
where $M_0=P(E(0))$,  $P(\cdot)$ is some polynomial function,  and $0<\alpha<1.$
By Young's inequality and adjusting the constants, one gets
$$
\sup\limits_{[0,t]}{E}(t)\leq {M}_0+CtP\Big(\sup\limits_{[0,t]}{E}\Big),
$$
which yields as in \cite{Coutand2}   a time of existence $T_1$ independent of $\varepsilon$,
as well as  an energy estimate on the time interval $(0,T_1)$ independent of $\varepsilon$ of the form:
\begin{equation}\label{E:8.1}\begin{split}
\sup\limits_{[0,t]}{E}(t)\leq 2 {M}_0.
\end{split}\end{equation}
{By the $\epsilon-$independent estimate \eqref{E:8.1}, there exists a subsequence of $v^\epsilon$
converging to $v$ in $L^2(0,T;H^2(I))$ with $\eta=x+\int_0^t v(x,s)ds.$ The standard compactness
arguments shows that $v$ is a solution to \eqref{E:2.7}.}
Thus, we can prove Theorem \ref{T:1.1} for $2\leq\gamma<3$ and $\gamma=\frac{3}{2}$.
Moreover, in the above process, we find that the method can be extended to all the cases of $1<\gamma<2$, thus
Theorem \ref{T:1.1} can be proved for all the general case $1<\gamma<3.$

Now, we are ready to prove the uniqueness of solutions.
For two solutions $v_1(x,t),v_2(x,t)$ satisfying Theorem \ref{T:1.1} to the free-boundary problem of the compressible Euler equations  \eqref{E:1.1}, we want to prove $v_1(x,t)=v_2(x,t)$.
{In fact, from \eqref{E:2.11}, there exits three positive constants $c_1,c_2,c_3$ such that
\begin{equation}\label{E:2.11a}c_1\leq \partial_x\eta_i(x,t)\le c_2, \ \ \text{and}\ \
|\partial_xv_i(x,t)|\leq c_3, \quad (x,t)\in[0,1]\times [0,T],\quad i=1,2.\end{equation}}
We define $\delta v=v_1-v_2$, then $\delta v$ satisfy the following equation:
\begin{equation}\label{E:9.1}
\omega^{1+2\mu}\partial_t\delta v+\Big[\omega^{2+2\mu}e^{S_0}\Big(\frac{1}{(\eta_1')^\gamma}
-\frac{1}{(\eta_2')^\gamma}\Big)\Big]'=0.
\end{equation}

By considering the fifth differential version of \eqref{E:9.1}, 
from the Proposition \ref{L:7.1}, we have the following equation for $\delta v$:
\begin{equation}\label{E:5.5}\begin{split}
\frac{1}{2}\int_{I} &\omega^{1+2\mu}\big|\partial_t^5 \delta v\big|^2+\frac{\gamma}{2}\int_I\big|\partial_t^4\delta v'\big|^2\frac{\omega^{2+2\mu}}
{(\eta_1')^{\gamma+1}}e^{S_0}+\epsilon\int_I\omega^{2+2\mu}|\partial_t^5\delta v'|^2 e^{S_0}\\
=&-\frac{(\gamma+1)\gamma}{2}\int_0^t\int_I\omega^{2+2\mu}e^{S_0} \Big(\frac{1}{(\eta_1')^{\gamma+2}}v_1'\big|\partial_t^4v_1'\big|^2 -\frac{1}{(\eta_2')^{\gamma+2}}v_2'\big|\partial_t^4v_2'\big|^2\Big)\\
&-\sum_{\alpha=1}^4 c_\alpha\int_0^t\int_I\omega^{2+2\mu}e^{S_0}\Big(\partial_t^\alpha\frac{1}{(\eta_1')^{\gamma+1}}
\partial_t^{4-\alpha}v_1'
\partial^5_t v_1'-\partial_t^\alpha\frac{1}{(\eta_2')^{\gamma+1}}
\partial_t^{4-\alpha}v_2'
\partial^5_t v_2' \Big)\\
&-\gamma\int_I\omega^{2+2\mu}e^{S_0}\Big(\frac{\partial_t^4 v_1'\partial_t^4 v_2'}{(\eta_1')^{\gamma+1}}
+\frac{1}{2}\big(\frac{1}{(\eta_1')^{\gamma+1}}+\frac{1}{(\eta_2')^{\gamma+1}}\big)|\partial_t^4v_2'|^2\Big)\\
&+2\epsilon\int_I\omega^{2+2\mu}\partial_t^5 v_1'\partial_t^5 v_2'e^{S_0}.\\
\end{split}\end{equation}
Using the energy estimate and weighted embedding inequality
\eqref{E:3.1a}, we obtain the analogous  version of  Proposition \ref{L:7.1} and Proposition \ref{L:5.3} for
$\delta v,$ and $\partial_t^4 \delta v\in L^\infty(0,T;L^2(I))$.
Next, we consider the elliptic estimates for higher-order spatial derivatives, $\delta v$ satisfies the following
equation
\begin{equation*}\label{E:9.2}\begin{split}
\big(\omega^{2+2\mu}\partial^k_t \delta v'e^{S_0}\big)'&-\frac{\varepsilon}{\gamma}\partial_t\big(\omega^{2+2\mu}\partial_t^k \delta v'
e^{S_0}\big)'=\delta g,\\
\end{split}\end{equation*}
where
\begin{equation*}\begin{split}
\delta g=&-\frac{1}{\gamma}\omega^{1+2\mu}\partial_t^{k+2}\delta v+\sum_{\alpha=1}^k\Big(c_\alpha\omega^{2+2\mu}e^{S_0}\big(\partial_t^\alpha
\frac{1}{(\eta_1')^{\gamma+1}}\partial^{k-\alpha}_tv_1'-\partial_t^\alpha
\frac{1}{(\eta_2')^{\gamma+1}}\partial^{k-\alpha}_tv_2'\big)\Big)'\\
&+\left[\left(1-\frac{1}{(\eta_1')^{\gamma+1}}\right)\big(
\partial_t^{k}v_1'e^{S_0}\big)'-\left(1-\frac{1}{(\eta_2')^{\gamma+1}}\right)\big(\omega^{2+2\mu}
\partial_t^{k}v_2'e^{S_0}\big)'\right]\\
&-(\gamma+1)\omega^{2+2\mu}e^{S_0}\Big(\frac{\partial_t^k v_1'\eta_1''}{(\eta'_1)^{\gamma+2}}
-\frac{\partial_t^k v_2'\eta_2''}{(\eta'_2)^{\gamma+2}}\Big).\\
\end{split}\end{equation*}
From  $\eta''_i=\int_0^t v_i'' dt$, using \eqref{E:2.11a} to control $\eta_i'$,  by the weighted embedding inequality
\eqref{E:3.1a} and repeating elliptic estimates, we have
$$\sup\limits_{t\in[0,T]}\sum_{s=0}^4
||\partial_t^{s}\delta v||^2_{2-s/2}\leq CT P\big(\sup\limits_{t\in[0,T]}\sum_{s=0}^4
||\partial_t^{s}\delta v||^2_{2-s/2}\big),$$
which shows that $\delta v=0.$

\bigskip

\section*{Acknowledgments}
Y. Geng's research was supported in part by the National Natural Science Foundation of China through grant 11201308 and the innovation program of Shanghai Municipal Education Commission (13ZZ136).
Y. Li's research was supported in part by National Natural Science Foundation of China through grants 11571232 and 11831011.
D. Wang's research was supported in part by the National Science Foundation under grants DMS-1312800 and DMS-1613213.
R. Xu's research was supported in part by the National Natural Science Foundation of China through grant 11471087.

\end{document}